%% file: baxter.tex
\documentclass[10pt,a4paper,reqno]{amsart} 

\usepackage{latexsym}
\usepackage{verbatim}
\usepackage{epsfig}
\usepackage{rotating}
\usepackage{amssymb}
\usepackage[latin1]{inputenc}
\usepackage{afterpage}
\usepackage{color}
\usepackage{dsfont}

\usepackage{amssymb,amsfonts,amsmath,stmaryrd}

\graphicspath{{Figures/}}

\catcode`\@=11
\def\section{\@startsection{section}{1}%
 \z@{.7\linespacing\@plus\linespacing}{.5\linespacing}%
 {\normalfont\bfseries\scshape\centering}}

\def\subsection{\@startsection{subsection}{2}%
  \z@{.5\linespacing\@plus\linespacing}{.5\linespacing}%
  {\normalfont\bfseries\scshape}}

\def\subsubsection{\@startsection{subsubsection}{3}%
 \z@{.5\linespacing\@plus\linespacing}{-.5em}
  {\normalfont\bfseries\itshape}}
\catcode`\@=12

\newcommand\mps[1]{\marginpar{\small\sf#1}}

\newtheorem{Theorem}{Theorem}
\newtheorem{Lemma}[Theorem]{Lemma}
\newtheorem{Proposition}[Theorem]{Proposition}
\newtheorem{Definition}[Theorem]{Definition}

\newtheorem{Remark}[Theorem]{Remark}

\def\qed{$\hfill{\vrule height 3pt width 5pt depth 2pt}$}

\newfont{\bbold}{msbm10 scaled \magstep1}
\newfont{\bbolds}{msbm7 scaled \magstep1}

\newcommand{\ns}{\mathbb{N}}

\newcommand{\rs}{\mathbb{R}}

\newcommand{\cT}{\mathcal T}
\newcommand{\cP}{\mathcal P}
\newcommand{\cS}{\mathcal S}

\newcommand{\Sn}{\mathfrak S}

\newcommand{\beq}{\begin{equation}}
\newcommand{\eeq}{\end{equation}}
\newcommand{\gf}{generating function}

\def\emm#1,{{\em #1}}
\newcommand{\p}{permutation}
\newcommand{\ps}{permutations}
\newcommand{\si}{\sigma}

\newcommand{\rev}{{\rm{rev}}}

\newcommand{\mir}{{\rm{mir}}}

\newcommand{\BS}{\mathcal{B}}
\newcommand{\OS}{\mathcal{O}}

 \def\bpi{\overline{\pi}}
 \def\wh{\widehat}

 \def\cTm{\mathcal{T}_m}

\begin{document}
\title
[Baxter permutations and plane bipolar orientations]
{Baxter permutations and plane bipolar orientations}

\author[N. Bonichon]{Nicolas Bonichon}
\address{N. Bonichon: LaBRI, Universit\'e Bordeaux 1, 351 cours de la
  Lib\'eration, 33405 Talence, France}
\email{nicolas.bonichon@labri.fr}
\author[M. Bousquet-M\'elou]{Mireille Bousquet-M\'elou}
\address{M. Bousquet-M\'elou: CNRS, LaBRI, Universit\'e Bordeaux 1, 
351 cours de la Lib\'eration, 33405 Talence, France}   
\email{mireille.bousquet@labri.fr}
\author[\'E. Fusy]{\'Eric Fusy}
\address{\'E. Fusy: Dept. Mathematics, Simon Fraser University, 8888 University Drive, Burnaby, BC, V5A 1S6, Canada}
\email{eric.fusy@inria.fr}

\thanks{The three authors are supported by the project \emph{GeoComp}
  of the ACI Masses de Donn\'ees and  by the French ``Agence Nationale
de la Recherche'', project SADA ANR-05-BLAN-0372.}


\begin{abstract}
We present
a simple bijection between Baxter permutations of size $n$
and plane bipolar orientations with $n$ edges. This bijection
translates several classical parameters of permutations (number of
ascents, right-to-left maxima, left-to-right 
minima...) into natural parameters of plane bipolar orientations
(number of vertices, degree of the sink, degree of the
source...), and has remarkable symmetry properties.
By specializing it to Baxter \ps \ avoiding the pattern 2413, we
obtain a bijection with non-separable planar maps.
A further specialization yields a bijection between \ps\ avoiding 2413 and
3142 and  series-parallel maps. 
\end{abstract}
\maketitle

\date{\today}


\section{Introduction}
In 1964, Glen Baxter, in an analysis context, introduced a class of
\ps\ that now bear his name~\cite{Bax:64}. 
A permutation $\pi$ of the symmetric group  $\Sn_n$ is 
\emm Baxter, if one cannot find
$i<j<k$ such that $\pi(j+1)<\pi(i)< \pi(k) <\pi(j)$ or $\pi(j)<\pi(k)
<\pi(i)<\pi(j+1)$.
 These \ps\ were first enumerated around
 1980~\cite{CGHK:78,Mal:79,Vie:81,CDV:86}. More recently, they have
 been studied  in a slightly different perspective, in the general
 and very active framework 
of \emm pattern avoiding, \ps~\cite{mbm-motifs,DG:96,Gir:93,GL:00}.  In
particular, the number of Baxter \ps\ of $\Sn_n$ having   
 $m$ ascents, $i$ left-to-right maxima and $j$ right-to-left maxima
 is known to be~\cite{Mal:79}:
\beq\label{baxter-enum}
 \frac{ij}{n(n+1)} {n+1 \choose m+1}  
\left[  {n-i-1 \choose n-m-2}{ n-j-1 \choose m-1}-
{n-i-1 \choose n-m-1}{ n-j-1 \choose m}\right].
\eeq

A few years ago, another Baxter, the physicist \emm Rodney, Baxter,
studied the sum of the 
Tutte polynomials $T_M(x,y)$ of non-separable planar maps $M$ having a fixed
size~\cite{Bax:01}. 
  He  proved
that the coefficient of $x^1y^0$ in $T_M(x,y)$,  summed over all
rooted non-separable planar maps  $M$ having $n+1$ edges, $m+2$ vertices,
 root-face of degree $i+1$ and a root-vertex of degree $j+1$  was given
by~\eqref{baxter-enum}.  
He was unaware that these numbers had been
met before (and bear his name), and that
the coefficient of  $x^1y^0$ in $T_M(x,y)$ is the number of \emm bipolar
orientations, of $M$~\cite{greene-zaslavsky,gebhard-sagan}.
This number is also, up to a sign, the derivative of the chromatic
polynomial of $M$, evaluated at 1~\cite{lass-orientations}.

This is an amusing coincidence --- are all Baxters doomed to invent
independently  objects that are counted by the same numbers? --- which
was first noticed in~\cite{mbm-motifs}. It is the aim of this paper to
explain it via  a direct bijection between Baxter \ps\ and plane bipolar
orientations. Our bijection is simple to
implement, non-recursive, and has a lot of structure: it translates
many natural statistics of \ps\ into natural statistics of maps, and
behaves well with respect to symmetries. When restricted to Baxter
\ps\ avoiding the pattern 2413, it specializes into a
bijection between these \ps\ and rooted non-separable planar
maps, which is related
 to a recursive description on these maps by
Dulucq \emm et al,.~\cite{DGW:96}. 
The key of  the proofs, and (probably) the reason why this bijection
has so much structure,   is the existence of \emm two isomorphic
 generating trees, for Baxter \ps\ and plane orientations.

It is not the first time that intriguing connections arise between pattern
avoiding \ps\ and planar maps. For instance, several families
of \ps, including two-stack sortable \ps,  are equinumerous with
non-separable maps~\cite{west-stacks,zeil-stacks,DGW:96,claesson}. 
Connected 1342-avoiding permutations are equinumerous with bicubic planar
maps~\cite{bona}, and 54321-avoiding involutions are equinumerous with
tree-rooted
maps~\cite{mullin,gouyou-tableaux,bernardi,mbm-motifs}. Finally,
a bijection that sends plane bipolar orientations to 3-tuples of
non-intersecting paths appears in~\cite{FPS:07}. These configurations
of paths are known to be in one-to-one correspondence with Baxter
\ps~\cite{DG:96}. More  bijections
between these three families of objects are presented
in~\cite{felsner-baxter}. However, our construction is the first
direct bijection between Baxter \ps\ and plane bipolar orientations. Moreover,
it has interesting properties (symmetries, specializations...) that
the bijections one may obtain by 
composing the bijections of~\cite{felsner-baxter} do not
have.

\medskip
Let us finish this introduction with the outline of the paper. After
some preliminaries in Section~\ref{sec:preliminaries}, we
present our main results in Section~\ref{sec:main}: we describe the
bijection $\Phi$, its inverse $\Phi^{-1}$, and state some of its properties. In
particular, we explain how it transforms statistics on \ps\ into statistics on
orientations (Theorem~\ref{thm:phibiject}), and how it behaves with
respect to symmetry (Proposition~\ref{prop:symmetries}). 
In Section~\ref{sec:trees} we introduce two \emm generating trees,, which
respectively describe a recursive construction of
Baxter \ps\ and  plane bipolar orientations. We observe that these
trees are isomorphic, which means that Baxter \ps\ and orientations
have a closely related recursive structure. In particular, this
isomorphism implies the existence of a (recursively defined) canonical
bijection
$\Lambda$ between these two classes of objects, the properties of
which reflect the properties of the trees 
(Theorem~\ref{thm:phibiject-canon}). We prove in Section~\ref{sec:proofs} that
our construction $\Phi$ coincides with the canonical bijection
$\Lambda$, and then (Section~\ref{sec:proofs-inverse}) that our description of
$\Phi^{-1}$ is correct.  Section~\ref{sec:spe} is devoted to the
study of two interesting specializations of $\Phi$, and the final
section presents possible developments of this work.

\section{Preliminaries}\label{sec:preliminaries}

\begin{figure}[b!] 
\begin{center}
{\input{Figures/sym8-perm.pstex_t}}
\caption{The symmetries of the square act on Baxter \ps.}
\label{fig:sym8-perm}
\end{center}
\end{figure}
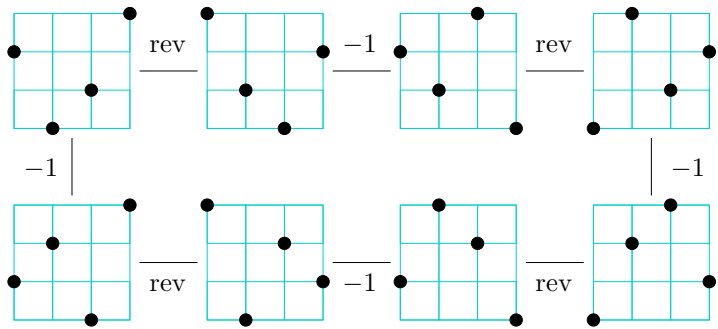

\subsection{Baxter permutations}\label{sec:baxter}
A permutation $\pi=\pi(1)\cdots\pi(n)$ will be represented by its \emm
diagram,, that is, the set of points 
$(i,\pi(i))$. Hence  the group of symmetries of the square, of order 8,
acts on \ps\ of  $\Sn_n$. This group is generated (for instance) by
the following two involutions:
\begin{itemize}
  \item[--] the inversion $\pi \mapsto \pi^{-1}$, which amounts to a reflection
around the first diagonal,
 \item[--] the reversion of \ps, which maps $\pi=\pi(1) \ldots \pi(n)$ on
$\rev(\pi)=\pi(n)\ldots\pi(1)$. 
\end{itemize}
It is clear from the definition given at the beginning of the
introduction that the reversion leaves invariant the set $\BS_n$ of Baxter \ps\
of size $n$. It is also true that the inversion leaves $\BS_n$
invariant, so that $\BS_n$ is invariant
under the 8 symmetries of the square (Fig.~\ref{fig:sym8-perm}). 
The invariance of  $\BS_n$ by inversion follows from  an alternative
description of Baxter \ps\ in terms of \emm barred, patterns~\cite{Gir:93},
which we now describe.  Given two permutations $\pi$ and 
$\tau=\tau_1\cdots \tau_k$,
an \emm occurrence of the pattern, $\tau$ in $\pi$ is a
subsequence $\pi(i_1), \ldots, \pi(i_k)$, with $i_1 <\ldots <i_k$,
which is order-isomorphic to $\tau$. If no such occurrence exists,
then $\pi$ \emm avoids, $\tau$. Now $\pi$ avoids the \emm  barred,
pattern 
$\tau=\tau_1\cdots \tau_{i-1}\overline{\tau_i}\tau_{i+1}\cdots\tau_k$ 
if every occurrence of
$\tau_1\cdots \tau_{i-1}\tau_{i+1}\cdots\tau_k$ 
in $\pi$ is a sub-sequence of an occurrence of 
$\tau_1\cdots \tau_{i-1}\tau_i\tau_{i+1}\cdots\tau_k$.
Then Baxter \ps\ are exactly the $25\bar314$- and $41\bar3 52$-avoiding \ps.

We equip $\rs^2$ with the natural product order:
$v\leq w$ if $x(v)\le x(w)$ and $y(v)\le y(w)$, that is, if $w$ lies
to the North-East of $v$. 
Recall that for two elements $v$ and $w$ of poset $\cP$, $w$ \emm covers, $v$ if $v< w$
and there is no $u$ such that $v< u <w$. The \emm Hasse diagram,
of $\cP$ is the digraph having vertex set $\cP$ and edges
corresponding to the covering relation. We orient these edges from
the smaller to the larger element.

Given a \p\ $\pi$, a \emph{left-to-right maximum} (or:
\emm lr-maximum,) is a
value $\pi(i)$ such that $\pi(i) > \pi(j)$ for all $j<i$. One defines
similarly rl-maxima, lr-minima and rl-minima. 

\subsection{Plane bipolar orientations}\label{sec:orientations}
A \emm  planar map, is a connected graph  embedded in the plane
with no edge-crossings, considered  up to continuous deformation. A
map has vertices, edges, and \emm faces,, which are the connected
components of $\rs^2$ remaining after deleting the edges. The \emm
outer face, is unbounded, the \emm inner faces, are bounded. 
The map is \emm separable, if there exists a vertex whose deletion
disconnects the map.
A \emph{plane bipolar orientation} $O$ is an acyclic orientation of a planar
map $M$ with a unique \emph{source}  $s$ (vertex with no ingoing edge)
and a unique \emph{sink} $t$ (vertex with no outgoing edge), both on
the outer face (Fig.~\ref{fig:bip}). The \emm poles, of $O$ are $s$
and $t$. One of the oriented paths going
from $s$ to $t$ has the outer face on its right: this path is the \emm
right border, of $O$, and its length is the  \emm right 
outer degree, of $O$. The \emm left outer degree, is
defined similarly.

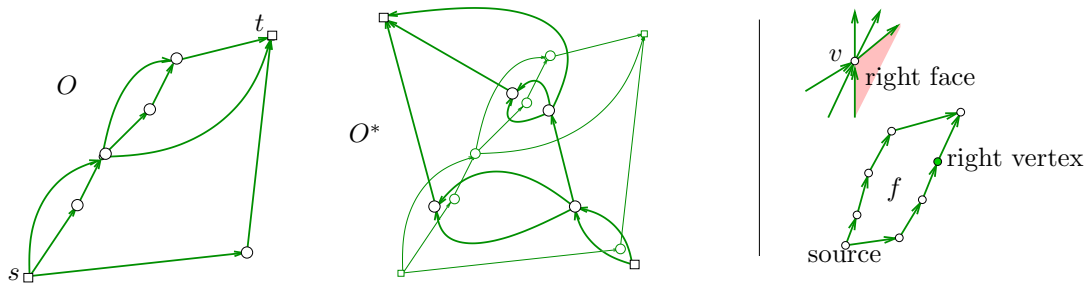
\begin{figure}[htb!] 
{\input{Figures/bip-new-ascent.pstex_t}}
\caption{Left: A plane bipolar orientation, having right (left) outer degree
2 (3), and its dual. Right: Local properties of plane bipolar orientations.}
\label{fig:bip}
\end{figure}

It can be seen that around each non-polar vertex $v$ of a plane
bipolar orientation, 
the edges are \emm sorted, into two blocks, one block of ingoing edges
and one block of outgoing edges: around  $v$, one finds
a sequence of outgoing edges, then a sequence of ingoing edges, and
that's it (Fig.~\ref{fig:bip}, right). 
 The face that is incident  to the last outgoing edge and the first
 ingoing edge (taken in clockwise order) is called the \emph{right 
face} of $v$.
  Symmetrically, the face that is incident  to the last
ingoing edge and the first outgoing edge (still in  clockwise order)
is called the \emph{left  face} of $v$. 
Dually, 
the border of every finite face $f$ contains exactly two maximal oriented paths
(Fig.~\ref{fig:bip}, right), forming a small bipolar orientation. Its
source (resp.~sink) is  called the \emm source, (resp.~\emm sink,) of
$f$. The other vertices of the face are respectively called \emm
right, and \emm left,  vertices of $f$: if $v$ is a right vertex of
$f$, then $f$ is   the left face of $v$.

A map $M$ is \emm rooted, if one of the edges adjacent to the outer
face is oriented, in such a way the outer face lies on its right. In
this case, a bipolar orientation of $M$ is required to have source $s$
and sink $t$, where $s$ and $t$ are the endpoints of the root edge. As
recalled above, the 
number of such orientations of $M$ is the coefficient of $x^1y^0$ in the
Tutte polynomial $T_M(x,y)$ (this is actually true of any graph $G$
with a directed edge). This number is non-zero if and only
if $M$ is non-separable~\cite{lempel}.

A map $M$ is \emm bipolar, if two of its vertices, lying on the outer
face and respectively called the \emm source, ($s$)
and the \emm sink, ($t$) are distinguished. In
this case, a bipolar orientation of $M$ is required to have source $s$
and sink $t$. There is of course a
simple one-to-one correspondence between rooted maps and bipolar maps,
obtained by deleting the root edge and taking its starting point
(resp.~ending point) as the source (resp.~sink). It will be convenient
(and, we hope, intuitive) to use the following notation:  if $M$ is
rooted,  $\check M$ is the bipolar map obtained by deleting the root
edge. If $M$ is bipolar, $\hat M$ is the rooted map obtained by adding a
root-edge. In this case,
$M$ admits a bipolar orientation 
if and only if $\hat M$ is non-separable.

Two natural transformations act on the set $\OS_n$ of (unrooted) plane bipolar
orientations having $n$ edges.  For a plane bipolar orientation $O$,
we define $\mir(O)$ to be the \emm mirror image, of $O$, that is, the
orientation  obtained by flipping $O$ around any line
(Fig.~\ref{fig:sym8-map}).   Clearly, $\mir$ is an involution.

The other transformation is \emm duality.,
The  dual plane orientation $O^*$ of $O$ is constructed as
shown on Fig.~\ref{fig:bip}. There is a vertex of $O^*$ in each bounded
face of $O$, and \emm two, vertices of $O^*$ (its poles) in the outer
face of $O$.
The edges of $O^*$  connect  the vertices
corresponding to faces of $O$ that are adjacent to a common edge, and
are oriented using the convention shown in the figure. Note
that $(O^*)^*$ is obtained by changing all edge directions in $O$, so
that the duality is a transformation of order 4. The
transformations $\mir$ and $O\mapsto O^*$ generate a group of
order 8 (Fig.~\ref{fig:sym8-map}). We shall see that our
bijection $\Phi$ allows us to superimpose Figs.~\ref{fig:sym8-perm}
and~\ref{fig:sym8-map}. That is, if $\Phi(\pi)=O$, then
$\Phi(\pi^{-1})$ is $\mir(O)$ and $\Phi(\rev(\pi))$ is $\mir(O^*)$.

\begin{figure}[htb!]
\begin{center}
{\input{Figures/sym8-map-ascent.pstex_t}}
\caption{A group of order 8 acts on plane bipolar orientations. The
  dashed edges join an orientation to its dual,  the others join
  mirror images.} 
\label{fig:sym8-map}
\end{center}
\end{figure}
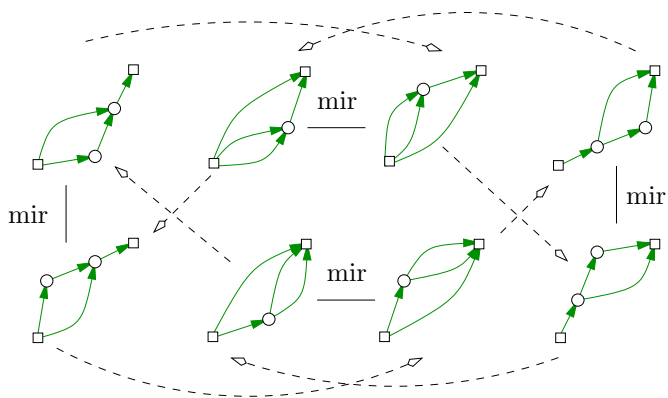

\section{Main results}\label{sec:main}
\subsection{From Baxter permutations to plane bipolar orientations}
\label{sec:bijection}
Let $\pi$ be a Baxter permutation. We first construct an embedded 
 digraph $\phi(\pi)$ with straight edges having black and white vertices.
The black vertices are the points $b_i=(i,\pi(i))$. The white vertices
correspond to the \emm ascents, of $\pi$. More precisely,
for each ascent $a$  (i.e., $\pi(a)<\pi(a+1)$), let 
$\ell_a=\max\{\pi(i):  i<a+1 \hbox{ and }  \pi(i)<\pi(a+1)\}$. The Baxter
property implies that for all $i > a$ such that $\pi(i)>\pi(a)$, one has
actually $\pi(i)>\ell_a$ (see Fig.~\ref{fig:saillant-rule}). Create
a white vertex  
$w_a=(a + 1/2 ,\ell_a + 1/2)$. Finally, add two more white vertices
$w_0=(1/2, 1/2)$ and $w_n=(n + 1/2, n + 1/2)$. 
(In other words, we consider that  $\pi(0) = 0$ and $\pi(n+1)= n+1$.)
We define the embedded digraph $\phi(\pi)$ to be the Hasse diagram of
this collection of black and white vertices for the product order on
$\rs^2$, drawn with straight edges (Fig.~\ref{fig:bijection}, middle).
Of course, all edges point to the North-East. 
We will prove in Section~\ref{sec:proofs} the following proposition.
\begin{Proposition}
 \label{prop:phi}
For all Baxter \p\ $ \pi$, the embedded graph $\phi(\pi)$ is  planar
(no edges cross),  bicolored
(every edge joins a black vertex and a white one), and every black vertex
has indegree and outdegree $1$.
\end{Proposition}
These properties can be
observed on the example of Fig.~\ref{fig:bijection}. 
Erasing all black 
vertices yields a plane bipolar orientation with 
source $w_0$ and sink $w_n$, which we define to be  $\Phi(\pi)$ (see
Fig.~\ref{fig:bijection}, right). 
Observe that  every point of the permutation gives rise to an edge of
 $\Phi(\pi)$: we will say that the point \emm corresponds, to the edge, and
vice-versa. We draw the attention of the reader on the fact that the
slightly vague word ``correspond'' has now a very precise meaning in
this paper.

\begin{figure}[htb!] 
\begin{center}
{\input{Figures/white-vertices-ascent.pstex_t}}
\caption{The insertion of white vertices in ascents.}
\label{fig:saillant-rule}
\end{center}
\end{figure}
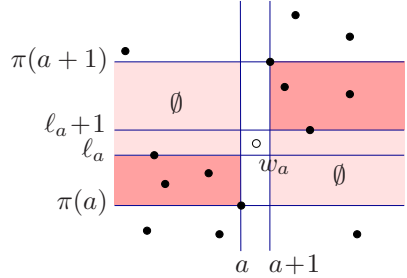

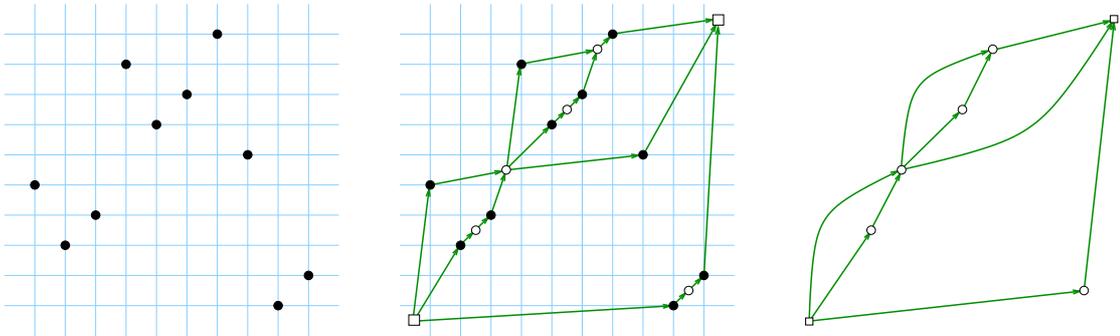
\begin{figure}[htb!]
\begin{center}
 \scalebox{0.25}{\input{Figures/bijection-new-ascent.pstex_t}}
\caption{The Baxter permutation 
$\pi=  5\, 3\, 4\, 9\, 7\, 8\, 1\!0\, 6\, 1\, 2$,
  the Hasse diagram $\phi(\pi)$  and the plane bipolar
  orientation $\Phi(\pi)$.}
\label{fig:bijection}
\end{center}
\end{figure}

\begin{Theorem}\label{thm:phibiject}
The map $\Phi$ is a bijection between Baxter permutations and plane
bipolar orientations, 
which transforms standard parameters as  follows:
\begin{center}
 \begin{tabular}
{r@{\hspace{0.2cm}}c@{\hspace{0.2cm}}l@{\hspace{0.2cm}}c@{\hspace{0.4cm}}|@{\hspace{0.4cm}}c@{\hspace{0.2cm}}r@{\hspace{0.2cm}}c@{\hspace{0.2cm}}l}
 size  & $\leftrightarrow$ & \#  edges,&  
 & &\# ascents & $\leftrightarrow$ &  \#   non-polar vertices,\\
   \#   lr-maxima & $\leftrightarrow$ & left outer degree,&
 &&\#   rl-minima & $\leftrightarrow$ & right outer degree,\\
 \#   rl-maxima & $\leftrightarrow$ & degree of the sink,& 
 &&\#   lr-minima & $\leftrightarrow$ & degree of the source.
 \end{tabular}
 \end{center}
\end{Theorem}

By Euler's formula, the number of inner
faces of $\Phi(\pi)$ is the number
of descents of $\pi$.  This theorem is proved in Section~\ref{sec:proofs}.
 It explains why the numbers~\eqref{baxter-enum} appear both in the
enumeration of Baxter \ps~\cite{Mal:79} and in the enumeration of bipolar
orientations~\cite{Bax:01}.

\subsection{The inverse bijection}\label{sec:inverse}

Let $O$ be a plane bipolar orientation, with  vertices colored white.
Let  $O'$ be  the bicolored oriented map obtained 
by inserting a black vertex in the middle of each edge of $O$
(Fig.~\ref{fig:inverse}, left). 
Recall how the in- and out-going edges are organized around the
vertices of $O$ (Fig.~\ref{fig:bip}, right). 
Let $T_x$ be  obtained from $O'$ by retaining, at each white vertex, 
only the first incoming edge 
in clockwise order
(at the sink, we only retain the incoming edge lying on the right border). 
All vertices of $T_x$, except from the source  of
$O$, have now  indegree one, and can be reached from the source. 
 Hence $T_x$ is a plane tree rooted at the source. 

Similarly, let $T_y$ be  obtained from $O'$ by 
retaining, at each white vertex,  
the last incoming edge in clockwise order (at the sink, we only
retain the incoming edge lying on the left border). 
Then $T_y$ is also  a spanning tree of  $O'$, rooted at the source of $O$. 

Label the black vertices of $T_x$ in prefix order, walking around
$T_x$ clockwise.  Label  the black vertices of $T_y$ in prefix order, but
walking around $T_y$ counterclockwise.
For every black vertex $v$ of $O'$, create a point
at coordinates $(x(v), y(v))$, where $x(v)$ and $y(v)$ are the labels
of $v$ in $T_x$ and $T_y$ respectively. Let $\Psi(O)$ be the collection
of points thus obtained (one per edge of $O$).

\begin{Theorem}
 For every plane bipolar orientation $O$, the set of points $\Psi(O)$
 is the diagram of the Baxter \p\ $\Phi^{-1}(O)$.
\end{Theorem}
 This can be checked on the examples of 
Figs.~\ref{fig:bijection}--\ref{fig:inverse}, and is proved in
Section~\ref{sec:proofs-inverse}. 
The construction $\Psi$ is closely related to an algorithm 
introduced in~\cite{DBTT:92} to draw  planar orientations on a grid.

\begin{figure}[htb!]
\begin{center}
 \scalebox{0.6}{\input{Figures/inverse-tggt-ascent.pstex_t}}
\caption{A bipolar orientation $O$, and  the trees  $T_x$ and $T_y$ used
  to compute the coordinates of the points of $\Psi(O)$.}
\label{fig:inverse}
\end{center}
\end{figure}
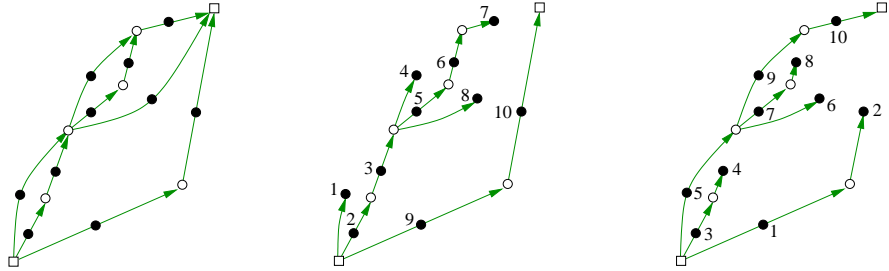

\subsection{Symmetries and specializations}\label{sec:sp-annonce}
\subsubsection{Symmetries}
We now describe how  the symmetries of the square, which act in a natural
way on Baxter \ps, are transformed through our bijection.

\begin{figure}[htb!]
\begin{center}
 \scalebox{1.2}{\input{Figures/symmetries-2-ascent.pstex_t}}
\caption{
The Hasse diagram and the bipolar orientations associated to
  $\pi^{-1}$ and $\rho(\pi)$, where 
$\pi= 5\, 3\, 4\, 9\, 7\, 8\, 10\, 6\, 1\, 2$
is the Baxter \p\ of Fig.~\ref{fig:bijection}.}
\label{fig:symmetries}
\end{center}
\end{figure}
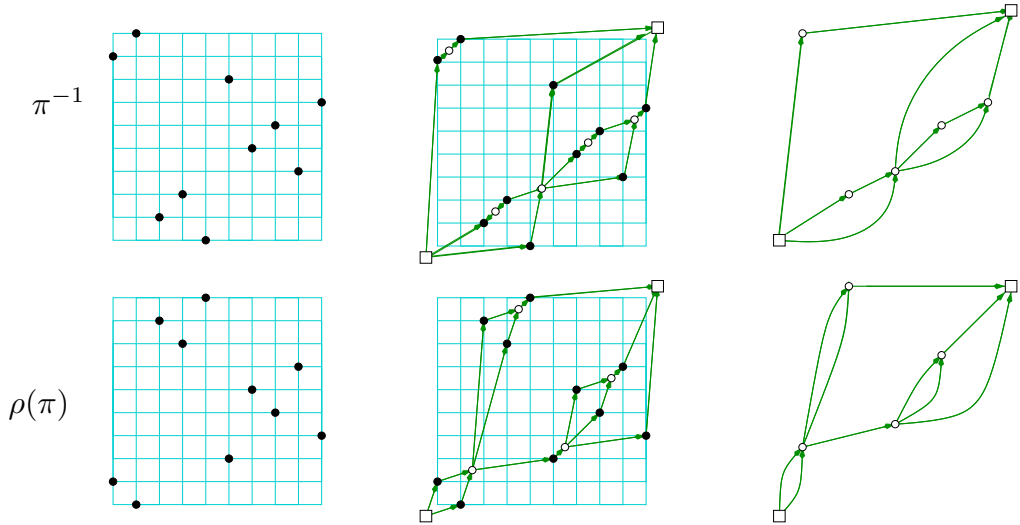
 \begin{Proposition}\label{prop:symmetries}
 Let $\pi$ be a Baxter \p\ and  $O=\Phi(\pi)$. Then
 $$
  \Phi(\pi^{-1})=\mir(O) \quad \hbox{and} \quad
 \Phi(\rev(\pi))=\mir(O^*).
 $$
By combining both properties, this implies
 $$
O^*=\Phi(\rho(\pi)),
$$
  where $\rho $ is the clockwise rotation by $90$ degrees.

Moreover, if the point $p=(i, \pi(i))$ of $\pi$ corresponds (via the
bijection $\Phi)$ to the edge $e$ of
$O$, then the point $(\pi(i),i)$ of $\pi^{-1}$ corresponds to the edge
$\mir(e)$ in $\mir(O)$, and the point $\rho(p)$ in $\rho(\pi)$
corresponds to the dual edge  of $e$ in $O^*$.
 \end{Proposition}
  The  symmetry properties dealing with the inverse and the rotation
 are illustrated in Fig.~\ref{fig:symmetries}.
 The first one  is easily proved from the definition of
  $\phi$ and $\Phi$. Indeed, it is clear from
  Fig.~\ref{fig:saillant-rule}  that $\pi$ has 
  an ascent at $a$ if and only if $\pi^{-1}$ has a ascent 
  at $\ell_a$, and that the white dots
  of the embedded graph 
  $\phi(\pi^{-1})$  are the symmetric of the white dots of $\phi(\pi)$
  with respect to the first diagonal. Of course this holds for black
  vertices as well. Hence the Hasse diagram $\phi(\pi^{-1})$ is
  obtained by flipping the diagram $\phi(\pi)$ around the first
  diagonal. The first property  follows, as well as the correspondence
  between the point $(\pi(i),i)$ and the edge $\mir(e)$.

The second property is non-trivial, and will be proved in
Section~\ref{sec:proofs}.

\subsubsection{Specializations}
%
As recalled in Section~\ref{sec:preliminaries}, a bipolar map $M$ has
a bipolar orientation 
if and only if the rooted
map $\hat M$ is non-separable. 

Take  a Baxter \p\ $\pi$ and the bipolar orientation $\Phi(\pi)$, with
poles $s$ and $t$. Let
$ M$ be the underlying bipolar map, and  define $\widehat
\Phi(\pi)$ to be the rooted non-separable $\hat M$.
We still call $s$ and $t$ the
source and the sink of the rooted map. It is not hard to see that
different \ps\ may give the same map. However,

\begin{Proposition}\label{prop:bijection_maps}
  The restriction of $\wh\Phi$ to Baxter \ps\  avoiding
  the pattern $2413$ (that is, to permutations avoiding $2413$ and
  $41\bar{3}52$) is a bijection between these \ps\  and rooted non-separable
  planar maps, which transforms standard parameters as  follows:
  \begin{itemize}
  \item[--] if $\pi$ has length $n$, $m$ ascents,  $i$ lr-maxima, $j$
    rl-maxima,   $k$ lr-minima and $\ell$ rl-minima, 
  \item[--]  then $\wh\Phi(\pi)$ has $n+1$
  edges, $m$ non-polar vertices, a sink of degree $j+1$, a source of
  degree $k+1$ and the face that lies to the right (resp.~left) of
    the root edge has   degree $i+1$ (resp.~$\ell+1$).
 \end{itemize}
\end{Proposition}
\noindent
This proposition is proved in Section~\ref{sec:spe-2413}.
The fact that permutations avoiding $2413$ and
  $41\bar{3}52$ are equinumerous with non-separable planar maps was
   already proved in~\cite{DGW:96}, by exhibiting  isomorphic \emm generating
   trees, for these two classes. This isomorphism  could be used to
   describe a recursive bijection between permutations and maps. It
   turns out that, up  to simple symmetries, our direct, non-recursive
   bijection, is   equivalent  to the one that is implicit
   in~\cite{DGW:96}. This is  explained in Section~\ref{sec:dulucq}.

Observe that, if $\pi$ is a Baxter \p\ avoiding  $2413$, then
$\pi^{-1}$ is a Baxter \p\ avoiding $3142$. As
$\Phi(\pi^{-1})=\mir(\Phi(\pi))$, the restriction of $\wh\Phi$ to
Baxter \ps\ avoiding  $3142$ is also a bijection with non-separable
planar maps. We now describe what happens when we restrict $\wh\Phi$
to \ps\ avoiding both 2413 and 3142  (such permutations are always
Baxter as  they obviously avoid the barred patterns $25\bar 3 14$ and
$41\bar 3 52$).

\begin{Proposition}\label{prop:SP}
  The restriction of $\wh\Phi$ to permutations  avoiding
  the patterns $2413$ and $3142$  is a bijection between these \ps\
  and rooted \emm series-parallel maps,,
which transforms the standard parameters in the same way as  in 
Proposition~\ref{prop:bijection_maps}.
\end{Proposition}

We say that a rooted non-separable map $M$ is 
\emm series-parallel,   if it does not contain  the complete graph $K_4$ as a minor.
 The terminology is slightly misleading: the map that can be
constructed recursively using the classical series and parallel
constructions (Fig.~\ref{fig:SP}) is not $M$ itself, but the bipolar  map
$\check M$. This is the case for the bipolar map of
Fig.~\ref{fig:bijection}, and one can check that the associated \p\
avoids both 2413 and 3142. Details are given in
Section~\ref{sec:spe-2}, with the proof of Proposition~\ref{prop:SP}.

It is easy to count series-parallel maps, using their recursive
description (we are in the simple framework of \emm decomposable
structures,~\cite{flajolet-sedgewick-book}). The series that counts
them by edges is (up to a factor $t$ accounting for the root edge) the
\gf\ of large Schr\"oder numbers~\cite[Exercise~6.39]{stanley-vol2},    
$$
S(t)= t \ \frac{1-t-\sqrt{1-6t+t^2}}2= \sum_{n\ge 1}
t^{n+1} 
\sum_{k=0} ^n \frac{(n+k)! }{k! (k+1)! (n-k)!}.
$$
Thus the number of \ps\ of size $n$ avoiding  $2413$ and $3142$ is the
$n$th Schr\"oder number, as was first proved in~\cite{West:95}. See
also~\cite{Gir:93}, where this result is proved via a bijection with
certain trees.

\section{Generating trees}\label{sec:trees}
A \emm generating tree, is a rooted plane tree with labelled nodes
satisfying the following property: if two nodes have the same label, the 
lists of labels of their children are the same. In other words,
the (ordered) list of labels of the children is completely determined by the
label of the parent. 
The rule that tells how to label the children of a node, given its
label, is called the 
\emm rewriting rule, of the tree. In particular, the tree $\cT$ with root
$(1,1)$ and rewriting rule:
\beq\label{gtij}
(i,j) \leadsto 
\left\{\begin{array}{lccl}
  (1,j+1),& (2,j+1),& \ldots &(i,j+1),\\
(i+1,j),& \ldots& (i+1,2),& (i+1, 1),
\end{array}\right.
\eeq
will be central in the proofs of our results.
The first three levels
of this tree are shown in Fig.~\ref{fig:GT-baxter}, left.
In this section, we describe a generating tree for Baxter \ps, and
another one for plane orientations, which are both isomorphic to
$\cT$. The properties of these trees, combined with this isomorphism,
imply the existence of a canonical bijection between Baxter \ps\
and plane orientations satisfying the conditions stated in
Theorem~\ref{thm:phibiject}. In the next section, we will prove that
this bijection coincides with our map $\Phi$.

\subsection{A generating tree for Baxter permutations}\label{sec:genbaxt}

Take a Baxter permutation $\pi$ of length $n+1$, and  remove the value
$n+1$: this gives another Baxter permutation, denoted $\bpi$, of length
$n$. Conversely, given $\si \in \BS_n$, it is well known, and easy to
see, that the \ps\ $\pi$ such that $\bpi=\sigma$ are obtained by
inserting the value $n+1$:
\begin{itemize}
  \item either just before an lr-maximum of $\si$,
\item or just after an rl-maximum of $\si$.
\end{itemize}
This observation was  used already in the first paper where Baxter \ps\ were
counted~\cite{CGHK:78}. We write $\pi=L_k(\si)$ if $\pi$ is obtain by
inserting $n+1$ just before the $k$th 
lr-maximum of $\si$, and
$\pi=R_k(\si)$ if $\pi$ is obtain by 
inserting $n+1$ just after the $k$th rl-maximum of $\si$ (with the
convention that the first lr-maximum is $\si(n)$). We henceforth
distinguish \emm left, and \emm right, insertions, or $L$- and
$R$-insertions for short. This is refined by talking about
$(L,k)$-insertions (and $(R,k)$-insertions)  when we need to specify
the position of the insertion.

This construction allows us to display Baxter \ps\ as the nodes of a
generating tree $\cT_b$. The root is the unique permutation of size $1$,
and the children of a node
$\si$  having $i$ lr-maxima and $j$ rl-maxima are, from left to right, 
$$
\begin{array}{lcccccccccc}
L_1(\si), &L_2(\si), &\ldots,& L_i(\si), 
& R_j(\si),& \ldots,& R_2(\si),& R_1(\si).
\end{array}
$$
Hence we find at level $n$ the \ps\ of $\BS_n$. 
The first layers of $\cT_b$ are shown on the right of Fig.~\ref{fig:GT-baxter}.
Observe that, if $\si$ has $i$ lr-maxima and $j$ rl-maxima, then
$L_k(\si)$ has $k$ lr-maxima and $j+1$ rl-maxima, while $R_k(\si)$ has
$i+1$ lr-maxima and $k$ rl-maxima. In other words, if we replace in
the tree $\cT_b$ every \p \ by the pair $(i,j)$ giving  the number of
lr-maxima and  rl-maxima, we obtain the generating tree $\cT$ defined
by~\eqref{gtij}. That is,
\begin{Proposition}
  The generating tree $\cT_b$ of Baxter \ps\ is isomorphic to the tree
  $\cT$ defined by~\rm{\eqref{gtij}}.
\end{Proposition}

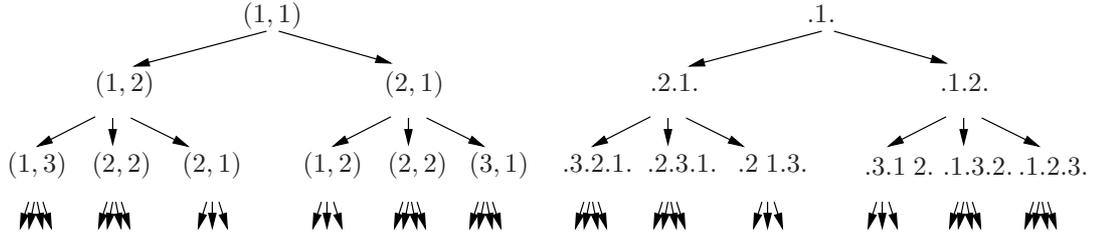
\begin{figure}[thb!]
\input{Figures/GT-label-new.pstex_t}
\caption{The generating tree $\cT$ with rewriting rule \eqref{gtij},
  and the  generating tree $\cT_b$ of Baxter permutations. The dots represent
  the possible insertion positions for the new maximal element.} 
\label{fig:GT-baxter}
\end{figure}

We now observe a property of the tree $\cT_b$ that will be crucial to
prove one of the symmetry properties of our bijection.
\begin{Proposition}\label{prop:reverse}
  The tree obtained from $\cT_b$ by replacing each \p\ $\pi$ by
  $\rev(\pi)$ coincides with the tree obtained by reflecting  $\cT_b$
  in a (vertical) mirror.
\end{Proposition}
In other words, if the sequence of insertions $(S_1, k_1), (S_2,
k_2),\ldots , (S_{n-1}, k_{n-1})$, where $S_i \in \{L,R\}$ and $k_i \in \ns$,
yields from the \p\  1 to the \p\ $\pi \in 
\BS_n$, then the sequence $(\bar S_1, k_1), (\bar S_2,
k_1), \ldots , (\bar S_{n-1}, k_{n-1})$ obtained by replacing all L's by R's
and vice-versa yields from 1 to $\rev(\pi)$.

\begin{proof}
  It suffices to observe that
$$
\rev(L_k(\pi))= R_k(\rev(\pi)) \quad \hbox{and} \quad 
\rev(R_k(\pi))= L_k(\rev(\pi)).
$$
\end{proof}

\subsection{A generating tree for plane bipolar orientations}\label{sec:genbip}

Let $O$ be a plane  bipolar orientation with $n+1$ edges.
Let $e$ be the last edge of the left border of $O$, and let $v$
be its starting point. The endpoint of $e$
is the sink $t$.
Perform the following transformation: 
  if $v$ has outdegree 1, contract $e$,  otherwise delete $e$. 
This gives a new  plane bipolar orientation, denoted $\overline O$,
having $n$ edges. Indeed, the
degree condition on $v$ guarantees that the contraction never creates
a cycle and that the deletion never creates a  sink.

Conversely, let $P$ be  an orientation with $n$ edges. We want to
describe  the orientations $O$ such that $\overline O=P$. Our
discussion is illustrated in Fig.~\ref{fig:bipolar-children}.
Let  $i$ be the left outer degree of $P$, and  $j$  the (in)degree of
the sink $t$ of $P$. Let $v_1=s,v_2, \ldots, v_i,v_{i+1}=t$ be the
vertices of the left border, visited from $s$ to $t$.  
Denote  $e_1,e_2,\ldots, e_j$ the edges incident to $t$, from
right to left (the infinite face is to the right of $e_1$ and to the
left of $e_j$). 
The orientations $O$ such that 
$\overline O=P$ are obtained by adding an edge $e$ whose contraction
or deletion gives $P$. This results in two types of edge-insertion:
\begin{itemize}
\item  {\bf  Type L.} For $k \in \llbracket 1,i\rrbracket  =\{1, 2,
  \ldots, i\}$, the
  orientation $L_k(P)$ is   obtained by adding  an edge 
from $v:=v_k$ to $t$, having the infinite   face on its left.
\item {\bf Type R.} 
For $k \in \llbracket 1,j\rrbracket$, the orientation $R_k(P)$ is
  constructed as follows. Split the vertex $t$ into two
  neighbour vertices $t$ and $v$,
and re-distribute the edges adjacent to $t$:  the edges
  $e_1, e_2,\dots, e_{k-1}$ 
remain connected to $t$, while $e_k, e_{k+1}, \dots, e_j$ are
  connected to $v$.  Add an edge from $v$ to $t$.
\end{itemize}
In both cases,  the last edge of the left border of the resulting
orientation $O$ joins $v$ to $t$.
After an $L$-insertion,  $v$ has outdegree 2 or more, so that $e$
will be deleted in the
construction of $\overline O$, giving the orientation $P$. 
After an $R$-insertion,  $v$ has outdegree 1, so that $e$
will be contracted in the construction of $\overline O$, giving the orientation $P$. The fact
that we use, as in the construction of Baxter \ps, the notation
$L_k$ and $R_k$ is of course not an accident.

\begin{figure}[htb!]
\begin{center}
\input{Figures/bipolar-children-new-ascent.pstex_t}
\caption{Inserting a new edge in a plane bipolar orientation ($i=3,j=4$).}
\label{fig:bipolar-children}
\end{center}
\end{figure}

We can now define the generating tree $\cT_o$ of plane bipolar orientations:
the root is the unique orientation with one edge, and the children of
a node $P$ having left outer degree $i$ and sink-degree $j$ are, from
left to right,  
$$
\begin{array}{lcccccccccc}
L_1(P), &L_2(P), &\ldots,& L_i(P), 
& R_j(P),& \ldots,& R_2(P),& R_1(P).
\end{array}
$$
Hence we find at level $n$ the orientations with $n$ edges. The first
levels of this tree are shown in Fig.~\ref{fig:GT-bipolar}. Observe
that, if $P$ has  left outer degree $i$ and sink-degree $j$, then
$L_k(P)$ has left outer degree  $k$ and sink-degree $i+1$, while 
$R_k(P)$ has left outer degree $i+1$ and sink-degree $k$.
In other words, if we replace in
the tree $\cT_o$ every orientation by the pair $(i,j)$ giving  the
left outer degree and the sink-degree, we obtain the generating tree
$\cT$ defined by~\eqref{gtij}.

\begin{Proposition}
  The generating tree $\cT_o$ of plane bipolar orientations is
  isomorphic to the tree   $\cT$ defined by~\rm{\eqref{gtij}}.
\end{Proposition}

\begin{figure}[htb!]
\begin{center}
\input{Figures/GT-bipolar-ascent.pstex_t}
\caption{The generating tree $\cT_o$ of plane bipolar orientations.}
\label{fig:GT-bipolar}
\end{center}
\end{figure}

We finally observe a property of the tree $\cT_o$ that is the
counterpart of Proposition~\ref{prop:reverse}. Recall  the definitions
of the dual and mirror orientations, given in Section~\ref{sec:orientations}.
\begin{Proposition}\label{prop:pol-*}
  The tree obtained from $\cT_o$ by replacing each orientation $O$ by
  $\mir(O^*)$ coincides with the tree obtained by reflecting  $\cT_o$
  in a (vertical) mirror.
\end{Proposition}
In other words, if the sequence of insertions $(S_1, k_1), (S_2,
k_2),\ldots , (S_{n-1}, k_{n-1})$, where $S_i \in \{L,R\}$ and $k_i \in \ns$,
yields from the root of $\cT_o$ to the orientation $O \in 
\OS_n$, then the sequence $(\bar S_1, k_1), (\bar S_2,
k_1), \ldots , (\bar S_{n-1}, k_{n-1})$ obtained by replacing all L's by R's
and vice-versa yields from the root to $\mir(O^*)$.
\begin{proof}
  It suffices to observe that
$$
\mir((L_k(O))^*)=R_k (\mir(O^*)) \quad \hbox{and} \quad
\mir((R_k(O))^*)=L_k (\mir(O^*)).
$$
This should be clear from Fig.~\ref{fig:bipolar-children-dual}, which
shows that applying $L_k$  to $O$ boils down to applying $R_k$  to
$\mir(O^*)$, and that conversely, applying $R_k$  to $O$ boils down to
applying $L_k$ to $\mir(O^*)$. Note that this figure only shows $O^*$, and not
its mirror image $\mir(O^*)$.
\end{proof}

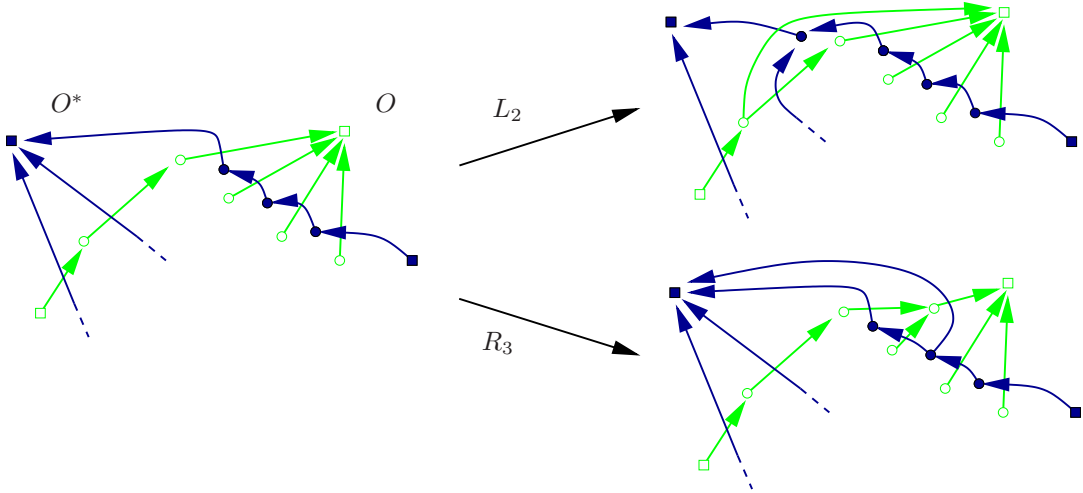
\begin{figure}[htb!]
\begin{center}
\input{Figures/bipolar-children-dual-ascent.pstex_t}
\caption{How the orientation $O^*$ is changed when $L_k$ or
  $R_k$ is applied to $O$.}
\label{fig:bipolar-children-dual}
\end{center}
\end{figure}

We conclude with an observation that will be useful in the proof of
Proposition~\ref{prop:symmetries}. 
\begin{Remark}\label{remark}
  Take  a plane bipolar orientation $O$ having $n$ edges, and label these
edges with $1,2, \ldots, n$ in the order where they were
created in the generating
 tree. Fig.~{\rm\ref{fig:bipolar-children-dual}} shows that when the
 edge $e$ is added to $O$ in the generating tree, the edge $\mir(e^*)$
 is added to $\mir(O^*)$ (we denote by $e^*$ the dual edge of
 $e$). Consequently, for all edges $e$ of $O$, the label of  $e$ in
 $O$ coincides with  the label of  $\mir(e^*)$ in $\mir(O^*)$.  
\end{Remark}

 \subsection{The canonical bijection}
We have seen that the generating trees $\cT_b$ and $\cT_o$ are both
isomorphic to the tree~\eqref{gtij} labelled by pairs
$(i,j)$. This gives immediately a canonical bijection $\Lambda$ between
Baxter \ps\ and plane bipolar 
orientations: this bijection maps the \p\ 1 to
the one-edge orientation, and is then defined recursively by
$$
\Lambda(L_k(\pi))= L_k(\Lambda(\pi)) \quad \hbox{and} \quad 
\Lambda(R_k(\pi))= R_k(\Lambda(\pi)).
$$
In other words,  if $\pi$ is obtained in the Baxter tree by the sequence
of insertions $(S_1, k_1), (S_2, 
k_2),\ldots , (S_{n-1}, k_{n-1})$, then $\Lambda(\pi)$ is the
orientation obtained by \emm the same sequence, of insertions in the tree of
orientations.
\begin{Theorem}\label{thm:phibiject-canon}
The map $\Lambda$ is a bijection between Baxter permutations and plane
bipolar orientations, which transforms standard parameters as  follows:
\begin{center} \begin{tabular}
{r@{\hspace{0.2cm}}c@{\hspace{0.2cm}}l@{\hspace{0.2cm}}c@{\hspace{0.4cm}}|@{\hspace{0.4cm}}c@{\hspace{0.2cm}}r@{\hspace{0.2cm}}c@{\hspace{0.2cm}}l}
 size  & $\leftrightarrow$ & \#  edges,&  
 & &\# ascents & $\leftrightarrow$ &  \#   non-polar vertices,\\
   \#   lr-maxima & $\leftrightarrow$ & left outer degree,&
 &&\#   rl-minima & $\leftrightarrow$ & right outer degree,\\
 \#   rl-maxima & $\leftrightarrow$ & degree of the sink,& 
 &&\#   lr-minima & $\leftrightarrow$ & degree of the source.
 \end{tabular}
 \end{center}
Moreover, if  $\pi$ is a Baxter \p\  and $O=\Lambda(\pi)$, then
$$
\Lambda(\rev(\pi))= \mir(0^*).
$$
\end{Theorem}
\begin{proof}
  The properties of $\Lambda$ dealing with the size, lr-maxima and rl-maxima
  follow directly from the isomorphism of the trees $\cT_b, \cT_o$ and
  $\cT$. The next three properties 
  are proved by observing that the relevant parameters  evolve in the
  same way in the recursive   construction of Baxter \ps\ and plane bipolar 
  orientations.  Indeed, if we replace every node $\pi$ of $\cT_b$ by
  $(i,j;m,k,\ell)$, where $i,j,m,k,\ell$ are respectively the number
  of lr-maxima, rl-maxima, ascents, lr-minima and rl-minima of $\pi$,
  we obtain the generating tree with root $(1,1;0;1,1)$ and rewriting
  rule:
$$
(i,j;m,k,\ell) \leadsto 
\left\{
\begin{array}{lccl}
  (1,j+1;m, k+1, \ell), (2,j+1;m,k,\ell ), \ldots, (i,j+1;m,k,\ell),\\
(i+1,j;m+1,k,\ell), \ldots, (i+1,2;m+1,k,\ell), (i+1, 1;m+1,k,\ell+1).
\end{array}
\right.
$$ 
Said in words,
 the number of ascents increases by 1 in an $R$-insertion, and
is unchanged otherwise. The number of lr-minima is only changed if we
perform an $(L,1)$-insertion (and then it increases by 1), and the
number of rl-minima is  only changed if we
perform an $(R,1)$-insertion (and then it increases by 1).

It is not hard to see that one obtains the same tree by replacing
every orientation $O$ of $\cT_o$ by  $(i,j;m,k,\ell)$, where
$i,j,m,k,\ell$ are respectively the left outer degree, the sink-degree, the number of non-polar vertices, the source-degree and
the right outer degree. That is, the number of vertices only increases (by
1) in an $R$-insertion, the degree of the source only increases (by 1) in
an $(L,1)$-insertion and the right outer degree only increases (by 1)
in an $(R,1)$-insertion.

\smallskip
Let us finally prove the  symmetry property. Let $(S_1, k_1), (S_2, 
k_2),\ldots , (S_{n-1}, k_{n-1})$ be the sequence
of insertions that leads  to $\pi$ in the Baxter
 tree. By definition of $\Lambda$, this sequence leads to
 $O=\Lambda(\pi)$ in the tree of orientations. By 
Propositions~\ref{prop:reverse} and~\ref{prop:pol-*}, the  sequence
$(\bar S_1, k_1), (\bar S_2,  
k_1), \ldots , (\bar S_{n-1}, k_{n-1})$ obtained by swapping the $L$'s and the $R$'s leads respectively to
$\rev(\pi)$ and  $\mir(O^*)$ in the trees $\cT_b$ and $\cT_o$. By
definition of $\Lambda$, this means  that  $\mir(O^*)=\Lambda(\rev(\pi))$.
\end{proof}

 \section{The mapping $\Phi$ is the canonical bijection}
\label{sec:proofs}

In the previous section, we have described recursively a bijection
$\Lambda$ that implements the isomorphism between the generating trees
$\cT_b$ and $\cT_o$,  and shown it has some interesting properties
(Theorem~\ref{thm:phibiject-canon}). We now prove that the mapping
$\Phi$ defined in Section~\ref{sec:main} coincides with this canonical
bijection $\Lambda$. Simultaneously, we prove the properties of the map $\phi$
stated in Proposition~\ref{prop:phi}.

 \begin{Proposition}\label{lem:phirealise}
 For each Baxter permutation $\pi$, the embedded graph $\phi(\pi)$ is
 planar and  bicolored, with black vertices of indegree and outdegree $1$.
 Moreover, $\Phi(\pi)=\Lambda(\pi)$.
 \end{Proposition}
 \begin{proof}
 The proof is by induction  on the size of $\pi$. Both statements are
 obvious for $\pi=1$. Now assume that the proposition holds for Baxter
permutations of size  $n$.  Let $\pi \in \BS_{n+1}$, and let  $\si$ be the
parent of $\pi$ in the tree $\cT_b$. Then either $\pi=L_k(\si)$, or
 $\pi=R_k(\si)$ for some $k$. Let $O=\Phi(\si)=\Lambda(\si)$.
We want to prove that  $\phi(\pi)$ satisfies the
   required conditions and that $\Phi(\pi)=L_k(O)$ (or that
 $\Phi(\pi)=R_k(O)$ in the case of an $R$-insertion). Both statements
   follow from a 
   careful observation of how $\phi(\si)$ is changed into $\phi(\pi)$ as
   $n+1$ is inserted in $\si$. Some readers will think that
   looking for two minutes at Figs.~\ref{fig:left-insertion-ascent}
 and~\ref{fig:right-insertion-ascent}  is sufficient to get 
   convinced of the result. For the others, we describe below in
   greater detail what happens during  the insertion of $n+1$ in $\si$.
Essentially,  we describe the   generating tree whose nodes are the
   embedded graphs $\phi(\pi)$,  for $\pi\in \BS_n$.

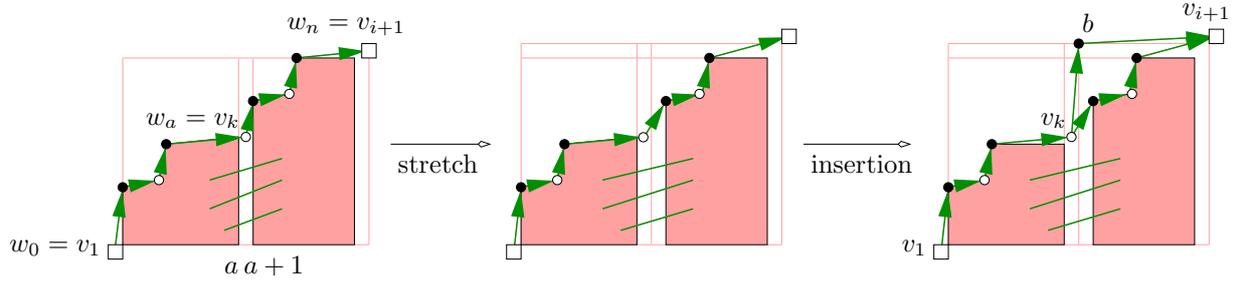
\begin{figure}
\begin{center}
\input{Figures/left-insertion-ascent.pstex_t}
\caption{How the graph $\phi(\si)$ changes during an $(L,k)$-insertion.}
\label{fig:left-insertion-ascent}
\end{center}
\end{figure}

\noindent
{\bf Case 1.} The \p\ $\pi$ is obtain by left insertion, that is,
   $\pi=L_k(\si)$.  Let $a+1$ be the abscissa of the $k$th lr-maximum of
   $\si$.
Let $i$ be the number of lr-maxima in $\si$. Then the left border of
$O$ has $i+1$ vertices, ranging from $v_1=w_0$ to $v_{i+1}=w_n$. The
vertex $w_a$ is the $k$th vertex $v_k$ of the left border of $O$.
As $n+1$ is inserted in $\si$, the ascents of $\si$ become ascents of
$\pi$,   and no  new ascent is created. 
All the vertices that occur in $\phi(\si)$ occur in
$\phi(\pi)$, but some of them are translated: the white vertex
$w_n$ moves from $(n +1/2, n + 1/2)$ to $(n+3/2,n+3/2)$, 
and all vertices
   located at abscissa $x\ge a+1$ move one unit to the right. These
   translations, illustrated by the 
   first two pictures of Fig.~\ref{fig:left-insertion-ascent}, stretch
  some edges but do not affect
 the covering    relations among vertices. We claim that they also do
 not affect planarity.
Indeed, it is easy to see that the Hasse diagram of  a set $\cS$  of points
 in the plane having distinct abscissas and ordinates, has no
 crossing if and only if $\cS$ avoids the pattern $21\bar3
 54$~\cite{mbm-butler}. Clearly, this property is not affected by
 translating to the right  the rightmost points of $\cS$.

After the translation operations, one new vertex is created: a black vertex $b=(a+1,n+1)$ 
corresponding to the new value $n+1$ in $\pi$.  We need to study how  this
affects the covering relations. That is, which vertices cover $b$, and which vertices are covered by  $b$? 

As  $b$ lies at ordinate  $n+1$, it is only covered by 
$w_n$. This results in a new bicolored edge
from $b$ to $w_n$, which lies sufficiently high not to affect the planarity. 
 Hence $b$  has outdegree 1. 

It remains to see which vertices $b$ covers. Clearly it covers
$w_a$. But then all the vertices lying to the South-West of $b$ are
actually smaller than $w_a$ for our ordering (because $\si(a)$ was an
lr-maximum), so that $b$ covers no vertex other than 
$w_a$. This means it has indegree 1.

To summarize, one goes from $\phi(\si)$ to $\phi(\pi)$, where $\pi=L_k(\si)$, by
 \begin{itemize}
   \item stretching some edges by a translation of  certain vertices, 
 \item inserting one new vertex, $b=(a+1,n+1)$,
 \item adding an edge from $w_a$ to $b$, and another one from $b$ to $w_n$.
 \end{itemize}

The resulting graph is still bicolored, with black vertices of
indegree and outdegree 1. The planarity is preserved as all the
changes  occur to the North-West of the lr-maxima of
$\si$. Finally, recall that the orientation $O=\Phi(\si)$ and $\Phi(\pi)$ are respectively
obtained by erasing the black vertices in $\phi(\si)$ and
$\phi(\pi)$. It should  now be clear that $\Phi(\pi)$ is exactly
the result of an $(L,k)$-insertion in $O$: denoting $v_1, \ldots, v_i,
v_{i+1}=w_n$ the vertices
 of the left border of $\Phi(\si)$, one has
simply added a new edge from $v_k=w_a$ to the sink $w_n$ in the outer face of
$\Phi(\si)$.

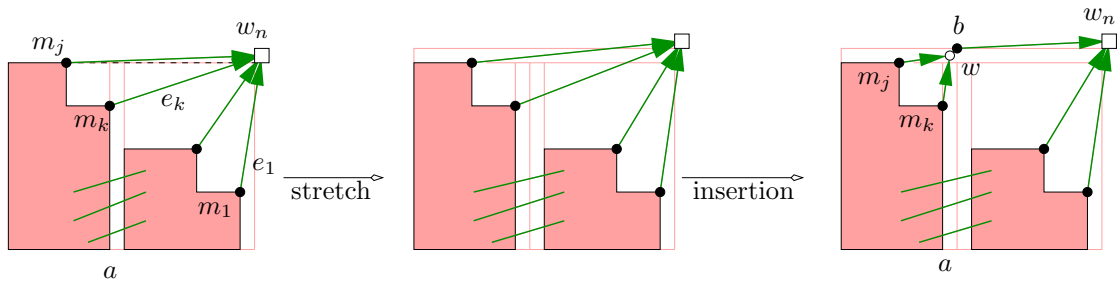
\begin{figure}
\begin{center}
\input{Figures/right-insertion-ascent.pstex_t}
\caption{How the graph $\phi(\si)$ changes during an $(R,k)$-insertion.}
\label{fig:right-insertion-ascent}
\end{center}
\end{figure}

\noindent {\bf Case 2.} The \p\ $\pi$ is obtain by right insertion, that is,
   $\pi=R_k(\si)$. 
Let $m_1, \ldots, m_j $ be the points of the diagram of $\si$
corresponding to its rl-maxima, from right to
left. Then the sink $t=w_n$ of $O$ has degree $j$, and for all $\ell$, the
$\ell$th edge that 
arrives at $t$ (from right to left) corresponds to the point $m_\ell$
of $\si$ via the correspondence $\Phi$. 
As noticed in the proof of Theorem~\ref{thm:phibiject-canon},
   the ascents of $\si$ become ascents of $\pi$, and a new ascent
   occurs at position $a$, if the $k$th rl-maximum of $\si$ is
   $\si(a)$. 
In particular, all the vertices that occur in $\phi(\si)$ occur in
$\phi(\pi)$. However, some of them are translated: the white vertex
$w_n$ moves from $(n+1/2,n+1/2)$ to $(n+3/2,n+3/2)$, while  the
vertices  located at abscissa $x\ge a+1$
move one unit to the right. These translations, illustrated by the
   first two pictures of Fig.~\ref{fig:right-insertion-ascent}, do not
   affect the planarity nor the covering   relations. 

Then two new vertices are created: a black vertex $b=(a+1,n+1)$
corresponding
 to the new value $n+1$ in $\pi$, and a white vertex $w=(a+1/2, 
n+1/2)$ corresponding to the new ascent. We need to study how  they
affect the covering relations. That is, which vertices cover $w$ or
$b$, and which vertices are covered by $w$ or $b$? 

As $w$ and $b$ lie respectively at ordinate $n+1/2$ and $n+1$, it is
easily seen that $w$ is only covered by $b$, which is only covered by
$w_n$. This results in two new bicolored edges, from $w$ to $b$, and
from $b$ to $w_n$, which lie sufficiently high not to affect the planarity. 
Moreover,  $b$ does not cover any vertex
other than $w$. We have thus proved that $b$ has indegree and
outdegree 1. 

It remains to see which vertices $w$ covers. These vertices were
covered in $\phi(\si)$ by vertices that are now larger than $w$. But
$w_n$ is the only vertex larger than $w$ that was already in
$\phi(\si)$. Since by assumption, $\phi(\si)$ is bicolored, the
vertices covered  by $w$ are black, and hence were rl-maxima in $\si$.
 The vertices $m_1, \ldots, m_{k-1}$ are still covered by
$w_n$ (they lie to the right of $w$), but $m_k, \ldots, m_j$ are
covered by $w$ (Fig.~\ref{fig:right-insertion-ascent}, right). 

To summarize, one goes from $\phi(\si)$ to $\phi(\pi)$, where $\pi=R_k(\si)$, by
 \begin{itemize}
   \item stretching some edges by a translation of  certain vertices, 
 \item inserting two new vertices,  $w=(a+1/2,n+1/2)$ and $b=(a+1,n+1)$,
 \item adding an edge from $w$ to $b$, and another one from $b$ to $w_n$,
 \item re-directing the edge that starts from $m_r$ so that it points to $w$
rather than $w_n$, for $r\ge k$.
 \end{itemize}

The resulting graph is still bicolored, with black vertices of
indegree and outdegree 1. The planarity is preserved as all the
changes  occur to the North-East of the rl-maxima of
$\si$. Finally, recall that the orientations $O=\Phi(\si)$ and $\Phi(\pi)$ are respectively
obtained by erasing the black vertices in $\phi(\si)$ and
$\phi(\pi)$. It should  now be clear that $\Phi(\pi)$ is exactly
the result of an $(R,k)$-insertion in $O$: the leftmost $j-k+1$ edges
that were pointing to $w_n$  are now pointing to the new white vertex $w$.
\end{proof}

We now know that $\Phi$ coincides with the canonical bijection
$\Lambda$, whose properties were stated in
Theorem~\ref{thm:phibiject-canon}. This implies
Theorem~\ref{thm:phibiject}, and the second symmetry property of
Proposition~\ref{prop:symmetries}.  The first property was
 proved just after the statement of this proposition, as well as
 the correspondence between the point $(\pi(i),i)$ and the edge
 $\mir(e)$.  To conclude the proof of
 Proposition~\ref{prop:symmetries}, it remains to prove that $\rho(p)$
 corresponds to the  edge $e^*$. 

Observe that,
as $n+1$ is inserted in $\si$ to form the \p\ $\pi$, the
 edge $e$ that is added to $O=\Phi(\si)$ to form  $\Phi(\pi)$
 corresponds to the point of ordinate $n+1$ in the diagram of
 $\pi$. Hence, the edge labelling introduced in Remark~\ref{remark}
 boils down to labelling every edge of  $\Phi(\pi)$ by the ordinate of the
 corresponding point of $\pi$. Consequently, Remark~\ref{remark} can be
 reformulated as follows:
  if the point $p=(i, \pi(i))$ corresponds to the edge $e$, then the
 point $\rev(p)$ (which is the point of $\rev(\pi)$ with the same
 ordinate as $p$) corresponds to $\mir(e^*)$. Combining this
 correspondence with the first one (which deals with $\pi^{-1}$ and
 $\mir(O)$)  gives the
 final statement of Proposition~\ref{prop:symmetries}.

\qed

 \section{The inverse bijection}
\label{sec:proofs-inverse}

It is now an easy task to prove that the map $\Psi$ described in
Section~\ref{sec:inverse}  is indeed the inverse of the bijection
$\Phi$. As we already know that $\Phi$ is bijective, it suffices to
prove that $\Psi(\Phi(\pi))=\pi $ for all Baxter \p\ $\pi$. For a
Baxter permutation $\pi$, 
we denote by $O$ the orientation $\Phi(\pi)$ and by $O'$ the bicolored
oriented map obtained from $O$ 
by adding a black vertex in the middle of each edge. 
Recall that  $\phi(\pi)$ is an embedding of $O'$. This allows us to
 consider the trees  $T_x$ and $T_y$ as  embedded in $\rs^2$.
In particular, the edge of $T_x$ joining  a white vertex $v$ to its parent 
 is the steepest  edge ending at $v$ in $\phi(\pi)$. See
Fig.~\ref{fig:inverse-proof}, left.

Every point of the diagram of $\pi$ corresponds\mps{détaillé} to an edge of
$O$. Thus the black vertices that we have added to form $O'$ are in
one-to-one correspondence with the points of $\pi$. This allows us to
identify the black vertices  of  $O'$ with the points of $\pi$. We want to check that the order induced on these
vertices by the clockwise prefix order of $T_x$ coincides with the order induced
by the abscissas of the points. Similarly, we want  to check that
the order induced on the
vertices by the counterclockwise prefix order of $T_y$ coincides with the order induced
by the ordinates  of the points. As the constructions of $\phi$ and
$\Psi$ are symmetric with respect to the first diagonal, it
suffices to prove the statement for the tree $T_x$. Since we are
comparing two total orders,  it
suffices to prove that if the vertex $v$ comes just after the vertex
$u$ in the prefix order of $T_x$, then $v$ lies to the right of $u$ in
 $\phi(\pi)$. Two cases occur.

If $u$ is not a leaf of $T_x$, it has a (unique, white) child, the first
child of which is $v$. As all edges in $\phi(\pi)$
point North-East, $v$ is to the right of $u$.

If $u$ is a leaf (Fig.~\ref{fig:inverse-proof}, right), let $w$ be its closest ancestor (necessarily white)
that has at least one child to the right of the branch leading to $u$.
By definition of the prefix order, the first of these children is $v$.  Observe that, by definition
of $\Psi$, the path of $T_x$ joining  $u$ to its ancestor $w$ is 
the \emm steepest, (unoriented) path of  $\phi(\pi)$ leading from $u$ to $w$. 
More precisely, it is obtained by starting at $u$, and choosing at
each time the steepest  down edge, until $w$ is reached. 
Assume $v$ is to the left of $u$. Then  it is also below $u$, and, as
$\phi(\pi)$ is a Hasse diagram,
there exist paths in $\phi(\pi)$ from $v$ to 
$u$. Take the steepest of these: that is, start from $u$, take the
steepest down edge that can be extended into a down path ending at
$v$, and iterate until $v$ is reached. Finally, add the edge $(w,v)$: this gives a
path joining  $w$ and  $u$ that is steeper than the one in $T_x$, a
contradiction. Hence $v$ lies to the right of $u$. \qed

\begin{figure}
\input{Figures/inverse-proof-ascent.pstex_t}
\caption{Left: The tree $T_x$ obtained from the Baxter \p\ 
$\pi=  5\, 3\, 4\, 9\, 7\, 8\, 1\!0\, 6\, 1\, 2$
%
Right: Why the prefix order and the abscissa order coincide.}
\label{fig:inverse-proof}
\end{figure}
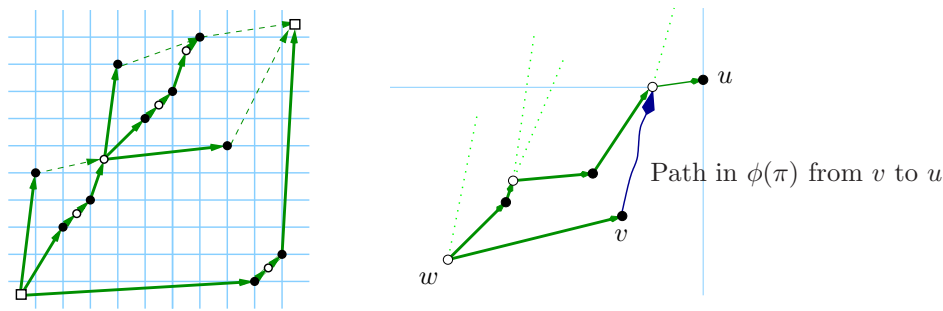

\section{Specializations}\label{sec:spe}

\subsection{Baxter \ps\ avoiding 2413 and rooted non-separable maps}
\label{sec:spe-2413}

The aim of this subsection is to prove
  Proposition~\ref{prop:bijection_maps}: if we restrict $\Phi$ to
  Baxter \ps\ avoiding 2413, add a root-edge from the source to the
  sink, and forget the orientations of all (non-root) edges, we obtain
  a bijection with rooted non-separable planar maps, which transforms
  standard parameters of \ps\ into standard parameters of maps.  Our first
  objective will be to describe the orientations corresponding via
  $\Phi$ to
  2413-avoiding Baxter \ps. Recall that the faces of a bipolar
  orientation have \emm left, and \emm right, vertices (Fig.~\ref{fig:bip}).
The following definition is illustrated in  	Fig.~\ref{fig:ROP}(a).
  \begin{Definition}
      Given a plane bipolar orientation $O$,  a \emph{right-oriented
	piece} (ROP) is
a $4$-tuple $(v_1,v_2,f_1,f_2)$ formed of two vertices $v_1$, $v_2$ and
	two inner faces $f_1$, $f_2$ of $O$ such that: 
\begin{itemize}
\item
 $v_1$ is the source of $f_1$ and is a left vertex of $f_2$,
\item
 $v_2$ is the sink of $f_2$ and is a right  vertex of $f_1$.
\end{itemize} 
A \emph{left-oriented piece} (LOP)
is defined similarly by swapping  'left' and  'right' in the
definition of a ROP. Consequently, a ROP in $O$ becomes a LOP in
$\mathrm{mir}(O)$ and vice-versa. 
  \end{Definition}

\begin{figure}[htb!]
\begin{center}
\input{Figures/ROP-ascent.pstex_t}
\caption{(a) A right-oriented piece (ROP) in a plane bipolar
  orientation. (b) The four distinguished edges 
of a ROP. (c) A minimal pattern 2413 in a Baxter permutation yields
  a ROP in the associated plane bipolar orientation. The dashed edges
  come from $\phi(\rho(\pi))$.} 
\label{fig:ROP}
\end{center}
\end{figure}
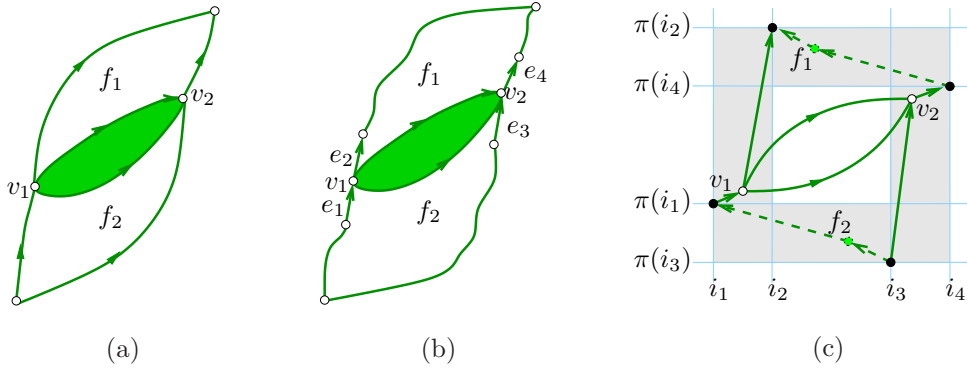

\begin{Proposition}[\cite{OdM:1994}]\label{fact:ROP}
Every  bipolar planar map admits  a unique bipolar orientation
with no ROP, and a unique plane bipolar orientation
with no LOP.
\end{Proposition}
\noindent The set of bipolar 
orientations of a fixed bipolar map can actually be equipped with
a structure of distributive
lattice, the minimum (resp.~maximum) of which is the unique orientation
with no ROP (resp.~LOP)~\cite{OdM:1994}.

We can now characterize the image by $\Phi$ of 2413-avoiding Baxter \ps.
\begin{Proposition}
\label{lem:ROP}
 A Baxter permutation $\pi$ contains the pattern $2413$  if and only
 if the  bipolar orientation $O=\Phi(\pi)$ contains a
 ROP. Symmetrically, $\pi$ contains the pattern $3142$  if and only
 if  $O$ contains a  LOP.
\end{Proposition}
Simple examples are provided by $\pi=25314$ and $\pi=41352$. The
corresponding orientations (with a root edge added) are those of
Fig.~\ref{fig:K4}.
\begin{proof}
Assume  $O$ contains a ROP $(v_1,v_2,f_1,f_2)$. Denote
$e_1,e_2,e_3,e_4$  the four edges  shown in Fig.~\ref{fig:ROP}(b). To each
of them corresponds a point $p_i$ of the diagram of $\pi$, with $1\le
i \le 4$. We will
prove that these points form an occurrence of 2413. Recall that the
 order of the abscissas and ordinates of these points is
obtained from the trees $T_x$ and $T_y$ defined in Section~\ref{sec:inverse}.
Denote by $P(e)$ the path of $T_x$ that joins  (the middle of) the
edge $e\in O$ to the source. By definition of $T_x$, the point of $\pi$
corresponding to $e$ occurs in $\pi$ to the left of the point corresponding to
$e'$ 
if and only if either $e$ lies on the path $P(e')$, or $P(e)$ is on the
left of $P(e')$ when the two paths meet. From this observation and the
configuration of a ROP, it is clear that the $x$-order of the points
$p_i$  is $p_1,p_2,p_3,p_4$. By similar arguments,
the $y$-order of these  points is $p_3,p_1,p_4,p_2$. 
Hence they form an occurrence of 2413.

Conversely, let $\pi$ be a Baxter permutation containing an occurrence
$p_1, p_2, p_3, p_4$ of the pattern 2413. This means that $p_j=(i_j,
\pi(i_j))$ with $i_1 <i_2 <i_3 <i_4$ and
$\pi(i_3)<\pi(i_1)<\pi(i_4)<\pi(i_2)$.
We can assume that the $p_i$ form a  \emm minimal, occurrence of 2413:
that is,  the bounding rectangle $R$ of the $p_i$ does not
contain any other occurrence of the pattern.  Then it is easy to see
that  no point of $R$ lies between columns
$i_1$ and $i_2$ (this would create a smaller pattern). Similarly, no
point of $R$ lies between columns $i_3$ 
and $i_4$, or between rows $\pi(i_3)$ and   $\pi(i_1)$, or between
rows $\pi(i_4)$ and   $\pi(i_2)$.  The empty areas are
shaded in Fig.~\ref{fig:ROP}(c). 

We now want to prove that $O$ contains a ROP. 
Consider  an oriented path
going from $p_1$ to $p_2$ in the embedded Hasse diagram $\phi(\pi)$. 
As $\phi(\pi)$ is bipartite and  there is no  point in $R$ between
columns $i_1$ and $i_2$,  this path has
length 2, and goes  from $p_1$ to $p_2$ via a unique white vertex, 
which we denote by $v_1$. As the $p_i$ have in- and out-degree 1, there is
no other path between $p_1$ and $p_2$.
Similarly, there is a unique oriented path
going from $p_3$ to $p_4$, which has length 2. Denote  by $v_2$ the (unique)
 white vertex of this path. 

We will exhibit a ROP whose vertices are  $v_1$ and $v_2$.
 To find the faces of this ROP, we consider the Baxter \p\
 $\pi^*=\rho(\pi)$ obtained by a clockwise rotation of $\pi$ by 90
 degrees. After this rotation, the points $\rho(p_i)$ still form a minimal
 occurrence of 2413. The above arguments imply that there exists in
 $\phi(\pi^*)$ a unique path from $\rho(p_3)$ to $\rho(p_1)$, which
 has length 2 and  contains only one white vertex $w_2$. Similarly,
there exists in
 $\phi(\pi^*)$ a unique path from $\rho(p_4)$ to $\rho(p_2)$, which
 has length 2 and  contains only one white vertex $w_1$. But
 $\Phi(\pi^*)$ is the dual orientation $O^*$
 (Proposition~\ref{prop:symmetries}). Let $f_1$ and  $f_2$ be the faces
 of $O$ that are the duals of $w_1$ and $w_2$.  We claim that $(v_1,
 v_2, f_1, f_2)$ is a ROP of $O$.

Consider the superposition of the edges of $\phi(\pi)$ and of the edges of
 $\phi(\pi^*)$, rotated by 90 degrees counterclockwise (dashed
 lines in Fig.~\ref{fig:ROP}(c)).
Let $e_i$ be the  edge of $O$ corresponding to
the point $p_i$, for $1\le i \le 4$. By
Proposition~\ref{prop:symmetries}, the dual edge $e_i^*$ corresponds
to the point $\rho(p_i)$. As 
$e_1^*$ starts from $w_2$, the edge $e_1$ lies on the left border of  the
face $f_2$ in $O$. 
 By considering the points $p_2, p_3$ and $ p_4$, one proves similarly
that $e_2$ lies on the left border of  $f_1$,  that $e_3$ lies on the
right border of $f_2$ and $e_4$ on the right border of $f_1$. In
particular, both $v_1$ and $v_2$ are  adjacent to $f_1$ and $f_2$.

As $p_4$ lies North-East of $p_1$, there is in $\phi(\pi)$ an oriented path
from $p_1$
to $p_4$. As $p_1$ has outdegree 1, the second vertex on this path is
$v_1$, and similarly, its next-to-last vertex is $v_2$. Thus the edge
$(v_1,p_2)$ of $\phi(\pi)$ is followed, in clockwise order around $v_1$, by another
outgoing edge: this implies that $v_1$ is the source of the face $f_1$. At the
other end of the path, 
we observe that the edge $(p_3,v_2)$ is
followed, in clockwise order around $v_2$, by another ingoing edge:
thus $v_2$ is the sink of $f_2$. 

The edge $e_1$ ends at $v_1$ and is on the left border of $f_2$:
as the sink of $f_2$ is $v_2$ (and $v_2\not = v_1$), $v_1$ is a
left vertex of $f_2$. Symmetrically, $v_2$ is a right vertex of
$f_1$. 

  Hence $(v_1,v_2,f_1,f_2)$ is a ROP. 

\medskip
It remains to prove the statement on 3142-avoiding Baxter \ps.
  Recall  that $\Phi(\pi^{-1})=\mir(O)$ and observe that
  $2413^{-1}=3142$. Note also that the map $\mir$ 
transforms right-oriented
pieces on left-oriented pieces. Consequently, a Baxter \p\ $\pi$
avoids 3142 if and only if the orientation $\Phi(\pi)$ has no
LOP.

\end{proof}

\noindent{\em Proof of  Proposition~\rm\ref{prop:bijection_maps}.} The restriction of $\Phi$ to
2413-avoiding Baxter \ps\ is a bijection between these \ps\ and
bipolar plane orientations with no ROP, which transforms standard
parameters as described in Theorem~\ref{thm:phibiject}. By Proposition~\ref{fact:ROP},
orientations with no ROP are in bijection with non-separable planar
maps (the bijection consists in adding a root-edge from the source to
the sink, and forgetting the orientation of all non-root edges). This
bijection increases by 1 the edge number, the degrees of the source
and the sink, and transforms the right and left outer degrees of the
orientation into the degrees of the faces lying, respectively, to the
left and right of the root-edge, minus
one. Proposition~\ref{prop:bijection_maps} follows. 

\subsection{Permutations avoiding 2413 and 3142,  and    series-parallel maps}
\label{sec:spe-2}
We now prove Proposition~\ref{prop:SP}. We have said that   a rooted
non-separable planar map $M$ is  \emph{series-parallel} if it does not
contain  $K_4$ as a minor. Let $\check M$ be the corresponding bipolar
map. We will
say that  $\check M$ itself is series-parallel.
By adapting the proof given for graphs in~\cite{BoGiKaNo07}, it is not
hard to see that a bipolar map  $\check M$ is series-parallel 
if and only if it can be constructed recursively,  starting from the 
single-edge map, by  applying a sequence of series and parallel
compositions:
\begin{itemize}
\item the \emm series composition, of two series-parallel bipolar maps $\check M_1$
  and $\check M_2$  is obtained by identifying the sink of $\check M_1$  with the
  source of $\check M_2$ (see Fig.~\ref{fig:SP} (b)),
\item the  \emm parallel composition, of 
  $\check M_1$ and $\check M_2$ is obtained by  putting $\check M_2$ to the right 
  of $\check M_1$,  and then identifying the sources of $\check M_1$ and $\check M_2$, 
  as well as their sinks (Fig.~\ref{fig:SP} (c)).
\end{itemize}

\begin{figure}[htb!]
\begin{center}
\input{Figures/SP.pstex_t}
\caption{The  operations that build all series-parallel bipolar maps: (a) taking  a single edge, (b) a series composition, (c) a parallel composition.}
\label{fig:SP}
\end{center}
\end{figure}
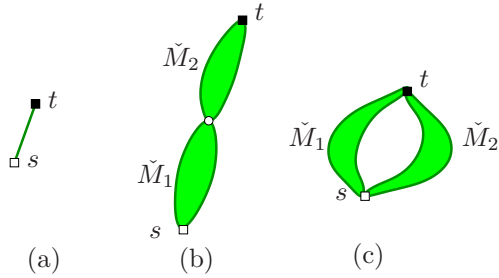

By Proposition~\ref{lem:ROP}, a Baxter \p\ $\pi$ avoids  both patterns
2413 and 3142 if and only if the corresponding plane bipolar
orientation has  no ROP nor LOP. We
have seen in Proposition~\ref{fact:ROP} that every bipolar map $\check M$
admits a unique orientation with no ROP and a unique orientation with
no LOP, which are respectively the minimal and maximal element of the
lattice of orientations of $\check M$. Thus $\check M$
admits an orientation with no ROP nor LOP if and only if it has a
unique bipolar orientation. Hence to prove Proposition~\ref{prop:SP}, it
suffices to establish the following lemma. 

\begin{Lemma}
\label{lem:SP}
A  bipolar map $\check M$ admits a unique bipolar orientation if
and only if it is series-parallel. 
\end{Lemma}
\begin{proof}
From the recursive  construction of  series-parallel  bipolar maps
shown in Fig.~\ref{fig:SP},
it is easily checked that such a map admits a unique
bipolar orientation (see also~\cite[Remark 6.2]{FOR:95}). 

Conversely, assume that $\check M$ is  not
series-parallel. This means that the corresponding rooted map $M$
contains $K_4$ as a minor.   
Observe that  $K_4$  admits exactly two bipolar orientations, show in
Fig.~\ref{fig:K4}.
From the \emph{Extension Lemma} of~\cite{FOR:95}, the two
bipolar orientations of $K_4$ can be extended to two distinct bipolar
orientations of $M$, and thus of $\check M$.  
\end{proof}

This concludes the proof of Proposition~\ref{prop:SP}.

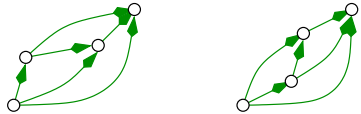
\begin{figure}[htb!]
\begin{center}
\input{Figures/K4.pstex_t}
\caption{The two bipolar orientations of $K_4$. The first one contains
a ROP, the other one a LOP.}
\label{fig:K4}
\end{center}
\end{figure}

\subsection{A link with a construction  of Dulucq, Gire  and West}
\label{sec:dulucq}
We have seen that the bijection $\Phi$ sends 
3142-avoiding Baxter \ps\
onto orientations with no LOP. Clearly, 3142-avoiding Baxter \ps\ form a
subtree of the generating tree $\cT_b$ of Fig.~\ref{fig:GT-baxter}: if $\pi$ avoids
3142, deleting its largest entry will not create an occurrence of this
pattern. As $\Phi$ is 
the canonical bijection between $\cT_b$ and the tree $\cT_o$ of
orientations, this implies that orientations with no LOP form a
subtree of $\cT_o$. 
%
In this subtree, replace every orientation by the
underlying bipolar map $\check M$, and then by the corresponding
non-separable rooted map $M$ (this means adding a root-edge to $\check
M$). This gives a generating tree $\cT_m$ for rooted non-separable maps.

In this section we give a description of this tree directly in terms of maps
(rather than orientations). We then  observe that, up to simple symmetries,
this is the tree that was introduced in~\cite{DGW:96} to prove that
non-separable planar maps are in bijection with permutations avoiding
the patterns 2413 and $41\bar{3}52$---equivalently, with 2413-avoiding Baxter
\ps. In accordance with~\cite{DGW:96}, where the children of a node of
the generating tree are not explicitly ordered, we will only describe
the non-embedded generating tree.  That is, we only explain how to
determine the parent of a given  rooted non-separable map.

Let $M$ be a rooted non-separable map. Recall that, by convention, the infinite
face lies to the right of the root-edge. Let   $e$ be the  edge
following the root-edge in counterclockwise order around the infinite
face. We say that $M$ is \emph{valid}  if $M\setminus e$ is
separable.

 \begin{Proposition}
   Let $M$ be a map of the tree $\cT_m$ having at least three edges,
   and let $e$ be defined as above. 
 The parent of $M$ in $\cTm$ is obtained by contracting $e$ if $M$ is
   valid, and  deleting it otherwise.
 \end{Proposition}
 \begin{proof}
   Given the description of the parent of an orientation $O$ in the
   tree $\cT_o$, and the bijection between orientations with no LOP and
   rooted non-separable maps, what
   we have to prove is the following statement: if $O$ is a
   bipolar orientation with no LOP, $e$ the last edge on its left
   border, $v$ the starting point of $e$,
and $M$ the
   associated rooted map, then $v$ has outdegree 1 in $O$ if and only
   if  $M$ is valid.

Assume first that  $v$ has outdegree 2 or more.  The parent
$\overline O$ of $O$ in $\cT_o$ is obtained by deleting $e$ from $O$. The
corresponding rooted map (obtained by adding to $\overline O$ a root-edge from $s$ to
$t$) is $M\setminus e$. It is non-separable (as it admits a bipolar
orientation). This means  that $M$ is not valid.

 Assume now that  $v$ has outdegree 1.  Let $f_1$
 be the  face lying to the right of  $e$ in $O$. If $f_1$ is the
 infinite face of $O$, then $M\setminus e$ is separable (the
 source is a separating vertex), so that $M$ is valid. Now assume that $f_1$ is
 a bounded face, and let $s_1$ be 
 its source. We need  the following lemma, the proof of which is
 delayed for the moment. 

\begin{Lemma}\label{lem:equiv}
The vertex $s_1$ is on the left outer border of $O$. 
\end{Lemma}
This lemma implies that $s_1$ is a separating vertex in $M\setminus
e$. Indeed, deleting $e$ and $s_1$ from $M$ disconnects the vertex $v$ from the
source (Fig.~\ref{fig:valid}). Hence $M$ is valid.
\end{proof}

\begin{figure}[htb!]
\begin{center}
\input{Figures/valid.pstex_t}
\caption{If $v$ has out degree 1, the rooted map $M$ is valid.}
\label{fig:valid}
\end{center}
\end{figure}
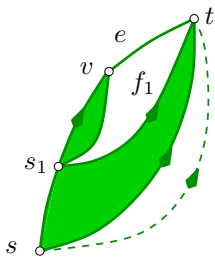

Up to minor changes (declaring $e$ as the root edge, and taking the
mirror of maps), we have recovered the generating tree $\cT_m$ of rooted
non-separable planar maps given in~\cite[Section~2.1]{DGW:96}. In that
paper, it is shown that $\cT_m$ is isomorphic to the generating tree
of 2413-avoiding Baxter \ps. The canonical (and recursive) bijection
between these \ps\ and non-separable planar maps that results from the
existence of  this isomorphism is not described explicitly
in~\cite{DGW:96}. Our paper gives a non-recursive description of this bijection
 (up to elementary symmetries), as a
special case of a more general correspondence. Note that the structure
of the tree $\cT_m$ is not as simple as that of the more general tree
$\cT_o$. In particular, the description given in~\cite{DGW:96}
involves an unbounded number of labels, while $\cT_o$ is  isomorphic
to a simple tree $\cT$ with two integers labels. 

\medskip
We still have to prove Lemma~\ref{lem:equiv}.

\noindent\emph{Proof of Lemma~\ref{lem:equiv}.}
Assume the lemma is wrong, and that $O$ is a minimal counterexample
(in terms of the edge number). 
By assumption,  $s_1$ is not on the left outer border of $O$. Thus the
left  face of $s_1$, called $f_2$, is finite, with source $s_2$ and
sink $t_2$ (Fig.~\ref{fig:equiv}(a)).
 The left border of $f_1$ is an oriented path $P_0$ going from $s_1$
 to $t$, which ends with the edge $e$. Let $P_1$ be the \emm leftmost,
 oriented path from $s_1$ to $t$:  at each vertex, it takes the
 leftmost outgoing edge. By planarity, the last edge of $P_1$ is also $e$.
  The first steps of $P_1$ follow the right border of  $f_2$,  from $s_1$ to $t_2$. 
 Let $O_1$ be the set of edges lying betwen $P_1\setminus e$ and
 $P_0\setminus e$ 
(both paths are included in $O_1$).   Then
$O_1$ is a bipolar orientation of source $s_1$ and sink $v$.
 
 Observe that $t_2$ is not incident to $f_1$, otherwise $(s_1,t_2,f_1,f_2)$
 would form a LOP. Hence the right  face of $t_2$, called  $f_3$,
 is a face of  $O_1$  (see Fig.~\ref{fig:equiv}(b), where only the
 orientation $O_1$ and the faces $f_1$ and $f_2$ are shown). Let $s_3$ be its source, and $t_3$ its sink.
As $s_1$ is the source of  $O_1$,  there exists an oriented path $P$
from $s_1$ to $s_3$. Consider the following two paths that go from
 $s_2$ to $t_3$: the first one follows the left 
 border of $f_2$ and then the portion of the left border of $f_3$ from
 $t_2$ to $t_3$; the second follows the right border of $f_2$ up to
 $s_1$, then the path $P$, and finally the right border of $f_3$. Let
 $O_2$ be the plane
 bipolar orientation  formed of the edges lying between these
 two paths (Fig.~\ref{fig:equiv}(c)).

The orientation $O_2$ has fewer edges than $O$, and we claim that it
is a counterexample to the lemma.  The last edge $e'$  of 
its left outer border  is also  
 the last edge of the left border of $f_3$. Hence its starting point lies between $t_2$ and $t_3$, and has outdegree $1$ in
 $O_2$. The face to the right of $e'$ is  $f_3$. 
Moreover, the source $s_3$ of $f_3$ is not on the left outer border of $O_2$, 
as it is separated from this path by the face $f_2$. 
 Hence $O_2$ is
a smaller counterexample than $O$, which yields a contradiction.     
 \qed

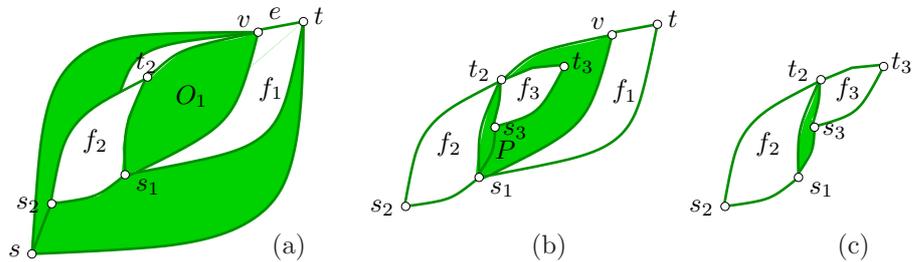
\begin{figure}[htb!]
\begin{center}
\input{Figures/equiv_LOP-new.pstex_t}
\caption{(a) A plane bipolar orientation with no LOP such that $v$ has
  outdegree 1. All edges are North-East oriented. The white areas are  faces.
The source
of $f_1$ has to be on the left outer border, otherwise, as shown in (b)-(c), a smaller counterexample can be produced.}
\label{fig:equiv}
\end{center}
\end{figure}

\section{Final comments}
It is natural to study the restriction of our bijection $\Phi$  (or
its inverse $\Phi^{-1}$) to interesting subclasses of permutations (or
orientations), for instance those that are counted by simple
numbers. This is the case for  alternating and
 doubly-alternating Baxter \ps~\cite{CDV:86,DG:96,GL:00}, or for
orientations of triangulations~\cite{tutte-12}. Recently, we have also
discovered a new and intriguing result, which deals with
fixed-point-free Baxter involutions. It is easy to construct for them
a generating tree, analogous to the tree $T_b$ of Baxter \ps\
(Fig.~\ref{fig:GT-baxter}). One goes from a node to its parent by deleting the
cycle containing the largest entry. Again, the tree is isomorphic
to a generating tree with two labels (like~\eqref{gtij}), that encode the number of
lr-maxima and rl-maxima. Using the techniques of~\cite{mbm-motifs}, we have
proved that the number
of such involutions of length $2n$ is
$$
\frac{3. \  2^{n-1}}{(n+1)(n+2)}{{2n}\choose n},
$$
which is also known to count Eulerian planar maps with $n$ edges.

After a 90 degrees rotation, Baxter involutions having no fixed point are
described by permutation diagrams that are symmetric with respect to
the second diagonal, and have no point on this diagonal. Via the
bijection $\Phi$, they correspond to  plane orientations
having a single source, but possibly several sinks, lying in the
outer face (Fig.~\ref{fig:involution}). It would be interesting to
have a combinatorial understanding of the above formula (and of the
various refinements of it that we have obtained).

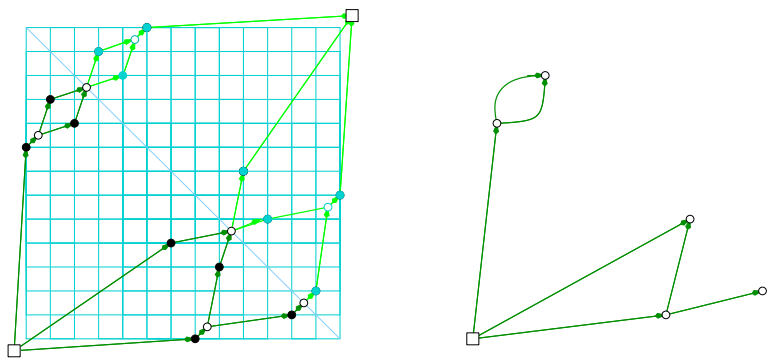
\begin{figure}[htb!]
\begin{center}
\input{Figures/involution-ascent.pstex_t}
\caption{A fixed-point-free  Baxter involution, rotated by 90 degrees,
  and the corresponding   plane orientation.}
\label{fig:involution}
\end{center}
\end{figure}

\bibliographystyle{plain}
\bibliography{baxter.bib}

\end{document}

%% file: Figures/sym8-perm.pstex_t
\begin{picture}(0,0)%
\includegraphics{sym8-perm.pstex}%
\end{picture}%
\setlength{\unitlength}{3158sp}%
\begingroup\makeatletter\ifx\SetFigFont\undefined%
\gdef\SetFigFont#1#2#3#4#5{%
  \reset@font\fontsize{#1}{#2pt}%
  \fontfamily{#3}\fontseries{#4}\fontshape{#5}%
  \selectfont}%
\fi\endgroup%
\begin{picture}(5512,2510)(4445,-2066)
\put(9601,-886){\makebox(0,0)[lb]{\smash{{\SetFigFont{10}{12.0}{\familydefault}{\mddefault}{\updefault}{\color[rgb]{0,0,0}\normalsize{$-1$}}%
}}}}
\put(8551, 89){\makebox(0,0)[lb]{\smash{{\SetFigFont{10}{12.0}{\familydefault}{\mddefault}{\updefault}{\color[rgb]{0,0,0}\normalsize{$\rev$}}%
}}}}
\put(7051, 89){\makebox(0,0)[lb]{\smash{{\SetFigFont{10}{12.0}{\familydefault}{\mddefault}{\updefault}{\color[rgb]{0,0,0}\normalsize{$-1$}}%
}}}}
\put(5551, 89){\makebox(0,0)[lb]{\smash{{\SetFigFont{10}{12.0}{\familydefault}{\mddefault}{\updefault}{\color[rgb]{0,0,0}\normalsize{$\rev$}}%
}}}}
\put(5551,-1786){\makebox(0,0)[lb]{\smash{{\SetFigFont{10}{12.0}{\familydefault}{\mddefault}{\updefault}{\color[rgb]{0,0,0}\normalsize{$\rev$}}%
}}}}
\put(7051,-1786){\makebox(0,0)[lb]{\smash{{\SetFigFont{10}{12.0}{\familydefault}{\mddefault}{\updefault}{\color[rgb]{0,0,0}\normalsize{$-1$}}%
}}}}
\put(8551,-1786){\makebox(0,0)[lb]{\smash{{\SetFigFont{10}{12.0}{\familydefault}{\mddefault}{\updefault}{\color[rgb]{0,0,0}\normalsize{$\rev$}}%
}}}}
\put(4576,-886){\makebox(0,0)[lb]{\smash{{\SetFigFont{10}{12.0}{\familydefault}{\mddefault}{\updefault}{\color[rgb]{0,0,0}\normalsize{$-1$}}%
}}}}
\end{picture}%

%% file: Figures/bip-new-ascent.pstex_t
\begin{picture}(0,0)%
\includegraphics{bip-new-ascent.pstex}%
\end{picture}%
\setlength{\unitlength}{987sp}%
\begingroup\makeatletter\ifx\SetFigFont\undefined%
\gdef\SetFigFont#1#2#3#4#5{%
  \reset@font\fontsize{#1}{#2pt}%
  \fontfamily{#3}\fontseries{#4}\fontshape{#5}%
  \selectfont}%
\fi\endgroup%
\begin{picture}(23490,6941)(7876,-5467)
\put(16051,-1936){\makebox(0,0)[lb]{\smash{{\SetFigFont{5}{6.0}{\familydefault}{\mddefault}{\updefault}{\color[rgb]{0,0,0}\normalsize {$O^*$}}%
}}}}
\put(30901,-2536){\makebox(0,0)[lb]{\smash{{\SetFigFont{5}{6.0}{\familydefault}{\mddefault}{\updefault}{\color[rgb]{0,0,0}\normalsize {right vertex}}%
}}}}
\put(29401,-3286){\makebox(0,0)[lb]{\smash{{\SetFigFont{5}{6.0}{\familydefault}{\mddefault}{\updefault}{\color[rgb]{0,0,0}\normalsize {$f$}}%
}}}}
\put(27451,-4936){\makebox(0,0)[lb]{\smash{{\SetFigFont{5}{6.0}{\familydefault}{\mddefault}{\updefault}{\color[rgb]{0,0,0}\normalsize {source}}%
}}}}
\put(28876,-511){\makebox(0,0)[lb]{\smash{{\SetFigFont{5}{6.0}{\familydefault}{\mddefault}{\updefault}{\color[rgb]{0,0,0}\normalsize {right face}}%
}}}}
\put(27976, 89){\makebox(0,0)[lb]{\smash{{\SetFigFont{5}{6.0}{\familydefault}{\mddefault}{\updefault}{\color[rgb]{0,0,0}\normalsize {$v$}}%
}}}}
\put(9301,-736){\makebox(0,0)[rb]{\smash{{\SetFigFont{5}{6.0}{\familydefault}{\mddefault}{\updefault}{\color[rgb]{0,0,0}\normalsize {$O$}}%
}}}}
\put(7876,-5386){\makebox(0,0)[rb]{\smash{{\SetFigFont{5}{6.0}{\familydefault}{\mddefault}{\updefault}{\color[rgb]{0,0,0}\normalsize {$s$}}%
}}}}
\put(13951,839){\makebox(0,0)[rb]{\smash{{\SetFigFont{5}{6.0}{\familydefault}{\mddefault}{\updefault}{\color[rgb]{0,0,0}\normalsize {$t$}}%
}}}}
\end{picture}%

%% file: Figures/sym8-map-ascent.pstex_t
\begin{picture}(0,0)%
\includegraphics{sym8-map-ascent.pstex}%
\end{picture}%
\setlength{\unitlength}{3158sp}%
\begingroup\makeatletter\ifx\SetFigFont\undefined%
\gdef\SetFigFont#1#2#3#4#5{%
  \reset@font\fontsize{#1}{#2pt}%
  \fontfamily{#3}\fontseries{#4}\fontshape{#5}%
  \selectfont}%
\fi\endgroup%
\begin{picture}(5073,3023)(3826,-2173)
\put(6301,-1261){\makebox(0,0)[lb]{\smash{{\SetFigFont{10}{12.0}{\familydefault}{\mddefault}{\updefault}{\color[rgb]{0,0,0}\normalsize{mir}}%
}}}}
\put(8626,-661){\makebox(0,0)[lb]{\smash{{\SetFigFont{10}{12.0}{\familydefault}{\mddefault}{\updefault}{\color[rgb]{0,0,0}\normalsize{mir}}%
}}}}
\put(6226, 89){\makebox(0,0)[lb]{\smash{{\SetFigFont{10}{12.0}{\familydefault}{\mddefault}{\updefault}{\color[rgb]{0,0,0}\normalsize{mir}}%
}}}}
\put(3826,-811){\makebox(0,0)[lb]{\smash{{\SetFigFont{10}{12.0}{\familydefault}{\mddefault}{\updefault}{\color[rgb]{0,0,0}\normalsize{mir}}%
}}}}
\end{picture}%

%% file: Figures/white-vertices-ascent.pstex_t
\begin{picture}(0,0)%
\includegraphics{Figures/white-vertices-ascent.pstex}%
\end{picture}%
\setlength{\unitlength}{1243sp}%
\begingroup\makeatletter\ifx\SetFigFont\undefined%
\gdef\SetFigFont#1#2#3#4#5{%
  \reset@font\fontsize{#1}{#2pt}%
  \fontfamily{#3}\fontseries{#4}\fontshape{#5}%
  \selectfont}%
\fi\endgroup%
\begin{picture}(5914,5490)(8491,-4999)
\put(8506,-2131){\makebox(0,0)[rb]{\smash{{\SetFigFont{5}{6.0}{\rmdefault}{\mddefault}{\updefault}\normalsize{$\ell_a\!+\!1$}}}}}
\put(10035,-1732){\makebox(0,0)[rb]{\smash{{\SetFigFont{5}{6.0}{\rmdefault}{\mddefault}{\updefault}\normalsize{$\emptyset$}}}}}
\put(12151,-2896){\makebox(0,0)[rb]{\smash{{\SetFigFont{5}{6.0}{\rmdefault}{\mddefault}{\updefault}\normalsize{$w_a$}}}}}
\put(8551,-3661){\makebox(0,0)[rb]{\smash{{\SetFigFont{5}{6.0}{\rmdefault}{\mddefault}{\updefault}\normalsize{$\pi(a)$}}}}}
\put(8551,-826){\makebox(0,0)[rb]{\smash{{\SetFigFont{5}{6.0}{\rmdefault}{\mddefault}{\updefault}\normalsize{$\pi(a+1)$}}}}}
\put(8506,-2626){\makebox(0,0)[rb]{\smash{{\SetFigFont{5}{6.0}{\rmdefault}{\mddefault}{\updefault}\normalsize{$\ell_a$}}}}}
\put(11341,-4921){\makebox(0,0)[rb]{\smash{{\SetFigFont{5}{6.0}{\rmdefault}{\mddefault}{\updefault}\normalsize{$a$}}}}}
\put(12736,-4921){\makebox(0,0)[rb]{\smash{{\SetFigFont{5}{6.0}{\rmdefault}{\mddefault}{\updefault}\normalsize{$a\!+\!1$}}}}}
\put(13250,-3161){\makebox(0,0)[rb]{\smash{{\SetFigFont{5}{6.0}{\rmdefault}{\mddefault}{\updefault}\normalsize{$\emptyset$}}}}}
\end{picture}%

%% file: Figures/bijection-new-ascent.pstex_t
\begin{picture}(0,0)%
\includegraphics{Figures/bijection-new-ascent.pstex}%
\end{picture}%
\setlength{\unitlength}{3315sp}%
\begingroup\makeatletter\ifx\SetFigFont\undefined%
\gdef\SetFigFont#1#2#3#4#5{%
  \reset@font\fontsize{#1}{#2pt}%
  \fontfamily{#3}\fontseries{#4}\fontshape{#5}%
  \selectfont}%
\fi\endgroup%
\begin{picture}(32989,9944)(-2721,-16283)
\end{picture}%

%% file: Figures/inverse-tggt-ascent.pstex_t
\begin{picture}(0,0)%
\includegraphics{inverse-tggt-ascent.pstex}%
\end{picture}%
\setlength{\unitlength}{4144sp}%
\begingroup\makeatletter\ifx\SetFigFont\undefined%
\gdef\SetFigFont#1#2#3#4#5{%
  \reset@font\fontsize{#1}{#2pt}%
  \fontfamily{#3}\fontseries{#4}\fontshape{#5}%
  \selectfont}%
\fi\endgroup%
\begin{picture}(8679,2642)(-911,-3043)
\end{picture}%

%% file: Figures/symmetries-2-ascent.pstex_t
\begin{picture}(0,0)%
\includegraphics{symmetries-2-ascent.pstex}%
\end{picture}%
\setlength{\unitlength}{1579sp}%
\begingroup\makeatletter\ifx\SetFigFont\undefined%
\gdef\SetFigFont#1#2#3#4#5{%
  \reset@font\fontsize{#1}{#2pt}%
  \fontfamily{#3}\fontseries{#4}\fontshape{#5}%
  \selectfont}%
\fi\endgroup%
\begin{picture}(13077,6774)(-164,-6148)
\put(151,-811){\makebox(0,0)[lb]{\smash{{\SetFigFont{5}{6.0}{\familydefault}{\mddefault}{\updefault}{\color[rgb]{0,0,0}\normalsize{$\pi^{-1}$}}%
}}}}
\put(-149,-4711){\makebox(0,0)[lb]{\smash{{\SetFigFont{5}{6.0}{\familydefault}{\mddefault}{\updefault}{\color[rgb]{0,0,0}\normalsize{$\rho(\pi)$}}%
}}}}
\end{picture}%

%% file: Figures/GT-label-new.pstex_t
\begin{picture}(0,0)%
\includegraphics{GT-label-new.pstex}%
\end{picture}%
\setlength{\unitlength}{987sp}%
\begingroup\makeatletter\ifx\SetFigFont\undefined%
\gdef\SetFigFont#1#2#3#4#5{%
  \reset@font\fontsize{#1}{#2pt}%
  \fontfamily{#3}\fontseries{#4}\fontshape{#5}%
  \selectfont}%
\fi\endgroup%
\begin{picture}(26148,5466)(286,-5569)
\put(301,-4036){\makebox(0,0)[lb]{\smash{{\SetFigFont{5}{6.0}{\rmdefault}{\mddefault}{\updefault}\normalsize{$(1,3)$}}}}}
\put(2401,-4036){\makebox(0,0)[lb]{\smash{{\SetFigFont{5}{6.0}{\rmdefault}{\mddefault}{\updefault}\normalsize{$(2,2)$}}}}}
\put(4651,-4036){\makebox(0,0)[lb]{\smash{{\SetFigFont{5}{6.0}{\rmdefault}{\mddefault}{\updefault}\normalsize{$(2,1)$}}}}}
\put(14101,-4036){\makebox(0,0)[lb]{\smash{{\SetFigFont{5}{6.0}{\rmdefault}{\mddefault}{\updefault}\normalsize{$.3.2.1.$}}}}}
\put(16201,-4036){\makebox(0,0)[lb]{\smash{{\SetFigFont{5}{6.0}{\rmdefault}{\mddefault}{\updefault}\normalsize{$.2.3.1.$}}}}}
\put(18451,-4036){\makebox(0,0)[lb]{\smash{{\SetFigFont{5}{6.0}{\rmdefault}{\mddefault}{\updefault}\normalsize{$.2\;1.3.$}}}}}
\put(23551,-4111){\makebox(0,0)[lb]{\smash{{\SetFigFont{5}{6.0}{\rmdefault}{\mddefault}{\updefault}\normalsize{$.1.3.2.$}}}}}
\put(25426,-4111){\makebox(0,0)[lb]{\smash{{\SetFigFont{5}{6.0}{\rmdefault}{\mddefault}{\updefault}\normalsize{$.1.2.3.$}}}}}
\put(21601,-4111){\makebox(0,0)[lb]{\smash{{\SetFigFont{5}{6.0}{\rmdefault}{\mddefault}{\updefault}\normalsize{$.3.1\; 2.$}}}}}
\put(20176,-286){\makebox(0,0)[lb]{\smash{{\SetFigFont{5}{6.0}{\rmdefault}{\mddefault}{\updefault}\normalsize{$.1.$}}}}}
\put(6151,-286){\makebox(0,0)[lb]{\smash{{\SetFigFont{5}{6.0}{\rmdefault}{\mddefault}{\updefault}\normalsize{$(1,1)$}}}}}
\put(7651,-4036){\makebox(0,0)[lb]{\smash{{\SetFigFont{5}{6.0}{\rmdefault}{\mddefault}{\updefault}\normalsize{$(1,2)$}}}}}
\put(9751,-4036){\makebox(0,0)[lb]{\smash{{\SetFigFont{5}{6.0}{\rmdefault}{\mddefault}{\updefault}\normalsize{$(2,2)$}}}}}
\put(11776,-4036){\makebox(0,0)[lb]{\smash{{\SetFigFont{5}{6.0}{\rmdefault}{\mddefault}{\updefault}\normalsize{$(3,1)$}}}}}
\put(23476,-2011){\makebox(0,0)[lb]{\smash{{\SetFigFont{5}{6.0}{\rmdefault}{\mddefault}{\updefault}\normalsize{$.1.2.$}}}}}
\put(16276,-2011){\makebox(0,0)[lb]{\smash{{\SetFigFont{5}{6.0}{\rmdefault}{\mddefault}{\updefault}\normalsize{$.2.1.$}}}}}
\put(9676,-2011){\makebox(0,0)[lb]{\smash{{\SetFigFont{5}{6.0}{\rmdefault}{\mddefault}{\updefault}\normalsize{$(2,1)$}}}}}
\put(2476,-2011){\makebox(0,0)[lb]{\smash{{\SetFigFont{5}{6.0}{\rmdefault}{\mddefault}{\updefault}\normalsize{$(1,2)$}}}}}
\end{picture}%

%% file: Figures/bipolar-children-new-ascent.pstex_t
\begin{picture}(0,0)%
\includegraphics{bipolar-children-new-ascent.pstex}%
\end{picture}%
\setlength{\unitlength}{1579sp}%
\begingroup\makeatletter\ifx\SetFigFont\undefined%
\gdef\SetFigFont#1#2#3#4#5{%
  \reset@font\fontsize{#1}{#2pt}%
  \fontfamily{#3}\fontseries{#4}\fontshape{#5}%
  \selectfont}%
\fi\endgroup%
\begin{picture}(14886,7825)(2115,-7217)
\put(2662,-3136){\makebox(0,0)[rb]{\smash{{\SetFigFont{5}{6.0}{\rmdefault}{\mddefault}{\updefault}\normalsize{$v_2$}}}}}
\put(7087,-1111){\makebox(0,0)[rb]{\smash{{\SetFigFont{5}{6.0}{\rmdefault}{\mddefault}{\updefault}\normalsize{$t$}}}}}
\put(2662,-4561){\makebox(0,0)[rb]{\smash{{\SetFigFont{5}{6.0}{\rmdefault}{\mddefault}{\updefault}\normalsize{$s=v_1$}}}}}
\put(16986,425){\makebox(0,0)[rb]{\smash{{\SetFigFont{5}{6.0}{\rmdefault}{\mddefault}{\updefault}\normalsize{$t$}}}}}
\put(14136,-250){\makebox(0,0)[rb]{\smash{{\SetFigFont{5}{6.0}{\rmdefault}{\mddefault}{\updefault}\normalsize{$v_3$}}}}}
\put(12561,-1600){\makebox(0,0)[rb]{\smash{{\SetFigFont{5}{6.0}{\rmdefault}{\mddefault}{\updefault}\normalsize{$v=v_2$}}}}}
\put(12561,-3025){\makebox(0,0)[rb]{\smash{{\SetFigFont{5}{6.0}{\rmdefault}{\mddefault}{\updefault}\normalsize{$v_1$}}}}}
\put(7876,-661){\makebox(0,0)[lb]{\smash{{\SetFigFont{5}{6.0}{\rmdefault}{\mddefault}{\updefault}\normalsize{Type L $(k=2)$}}}}}
\put(7876,-5386){\makebox(0,0)[lb]{\smash{{\SetFigFont{5}{6.0}{\rmdefault}{\mddefault}{\updefault}\normalsize{Type R $(k=3)$}}}}}
\put(5401,-1486){\makebox(0,0)[rb]{\smash{{\SetFigFont{5}{6.0}{\rmdefault}{\mddefault}{\updefault}\normalsize{$e_4$}}}}}
\put(5701,-2011){\makebox(0,0)[rb]{\smash{{\SetFigFont{5}{6.0}{\rmdefault}{\mddefault}{\updefault}\normalsize{$e_3$}}}}}
\put(15981,-4045){\makebox(0,0)[rb]{\smash{{\SetFigFont{5}{6.0}{\rmdefault}{\mddefault}{\updefault}\normalsize{$t$}}}}}
\put(14781,-4495){\makebox(0,0)[rb]{\smash{{\SetFigFont{5}{6.0}{\rmdefault}{\mddefault}{\updefault}\normalsize{$v$}}}}}
\put(4201,-1786){\makebox(0,0)[rb]{\smash{{\SetFigFont{5}{6.0}{\rmdefault}{\mddefault}{\updefault}\normalsize{$v_3$}}}}}
\put(6226,-2536){\makebox(0,0)[rb]{\smash{{\SetFigFont{5}{6.0}{\rmdefault}{\mddefault}{\updefault}\normalsize{$e_2$}}}}}
\put(6826,-3061){\makebox(0,0)[rb]{\smash{{\SetFigFont{5}{6.0}{\rmdefault}{\mddefault}{\updefault}\normalsize{$e_1$}}}}}
\put(15751,-6061){\makebox(0,0)[rb]{\smash{{\SetFigFont{5}{6.0}{\rmdefault}{\mddefault}{\updefault}\normalsize{$e_1$}}}}}
\put(14176,-5161){\makebox(0,0)[rb]{\smash{{\SetFigFont{5}{6.0}{\rmdefault}{\mddefault}{\updefault}\normalsize{$e_3$}}}}}
\put(13951,-4486){\makebox(0,0)[rb]{\smash{{\SetFigFont{5}{6.0}{\rmdefault}{\mddefault}{\updefault}\normalsize{$e_4$}}}}}
\put(15076,-5461){\makebox(0,0)[rb]{\smash{{\SetFigFont{5}{6.0}{\rmdefault}{\mddefault}{\updefault}\normalsize{$e_2$}}}}}
\end{picture}%

%% file: Figures/GT-bipolar-ascent.pstex_t
\begin{picture}(0,0)%
\includegraphics{GT-bipolar-ascent.pstex}%
\end{picture}%
\setlength{\unitlength}{1381sp}%
\begingroup\makeatletter\ifx\SetFigFont\undefined%
\gdef\SetFigFont#1#2#3#4#5{%
  \reset@font\fontsize{#1}{#2pt}%
  \fontfamily{#3}\fontseries{#4}\fontshape{#5}%
  \selectfont}%
\fi\endgroup%
\begin{picture}(12430,6055)(543,-7369)
\end{picture}%

%% file: Figures/bipolar-children-dual-ascent.pstex_t
\begin{picture}(0,0)%
\includegraphics{bipolar-children-dual-ascent.pstex}%
\end{picture}%
\setlength{\unitlength}{1579sp}%
\begingroup\makeatletter\ifx\SetFigFont\undefined%
\gdef\SetFigFont#1#2#3#4#5{%
  \reset@font\fontsize{#1}{#2pt}%
  \fontfamily{#3}\fontseries{#4}\fontshape{#5}%
  \selectfont}%
\fi\endgroup%
\begin{picture}(16724,7677)(824,-6974)
\put(8221,-4756){\makebox(0,0)[lb]{\smash{{\SetFigFont{5}{6.0}{\rmdefault}{\mddefault}{\updefault}\normalsize{$R_3$}}}}}
\put(2046,-1036){\makebox(0,0)[rb]{\smash{{\SetFigFont{5}{6.0}{\rmdefault}{\mddefault}{\updefault}\normalsize{$O^*$}}}}}
\put(6901,-1036){\makebox(0,0)[rb]{\smash{{\SetFigFont{5}{6.0}{\rmdefault}{\mddefault}{\updefault}\normalsize{$O$}}}}}
\put(8251,-1111){\makebox(0,0)[lb]{\smash{{\SetFigFont{5}{6.0}{\rmdefault}{\mddefault}{\updefault}\normalsize{ $L_2$}}}}}
\end{picture}%

%% file: Figures/left-insertion-ascent.pstex_t
\begin{picture}(0,0)%
\includegraphics{left-insertion-ascent.pstex}%
\end{picture}%
\setlength{\unitlength}{2368sp}%
\begingroup\makeatletter\ifx\SetFigFont\undefined%
\gdef\SetFigFont#1#2#3#4#5{%
  \reset@font\fontsize{#1}{#2pt}%
  \fontfamily{#3}\fontseries{#4}\fontshape{#5}%
  \selectfont}%
\fi\endgroup%
\begin{picture}(11730,2961)(61,-2389)
\put(7426,-1261){\makebox(0,0)[lb]{\smash{{\SetFigFont{7}{8.4}{\familydefault}{\mddefault}{\updefault}{\color[rgb]{0,0,0}\normalsize{insertion}}%
}}}}
\put(3151,-1261){\makebox(0,0)[lb]{\smash{{\SetFigFont{7}{8.4}{\familydefault}{\mddefault}{\updefault}{\color[rgb]{0,0,0}\normalsize{stretch}}%
}}}}
\put(10355,208){\makebox(0,0)[rb]{\smash{{\SetFigFont{7}{8.4}{\familydefault}{\mddefault}{\updefault}{\color[rgb]{0,0,0}\normalsize{$b$}}%
}}}}
\put(10075,-756){\makebox(0,0)[rb]{\smash{{\SetFigFont{7}{8.4}{\familydefault}{\mddefault}{\updefault}{\color[rgb]{0,0,0}\normalsize{$v_k$}}%
}}}}
\put(2176,-2311){\makebox(0,0)[rb]{\smash{{\SetFigFont{7}{8.4}{\familydefault}{\mddefault}{\updefault}{\color[rgb]{0,0,0}\normalsize{$a+1$}}%
}}}}
\put(1501,-2311){\makebox(0,0)[rb]{\smash{{\SetFigFont{7}{8.4}{\familydefault}{\mddefault}{\updefault}{\color[rgb]{0,0,0}\normalsize{$a$}}%
}}}}
\put(1513,-756){\makebox(0,0)[rb]{\smash{{\SetFigFont{7}{8.4}{\familydefault}{\mddefault}{\updefault}{\color[rgb]{0,0,0}\normalsize{$w_a=v_k$}}%
}}}}
\put( 76,-2086){\makebox(0,0)[rb]{\smash{{\SetFigFont{7}{8.4}{\familydefault}{\mddefault}{\updefault}{\color[rgb]{0,0,0}\normalsize{$w_0=v_1$}}%
}}}}
\put(11776,389){\makebox(0,0)[rb]{\smash{{\SetFigFont{7}{8.4}{\familydefault}{\mddefault}{\updefault}{\color[rgb]{0,0,0}\normalsize{$v_{i+1}$}}%
}}}}
\put(3226,239){\makebox(0,0)[rb]{\smash{{\SetFigFont{7}{8.4}{\familydefault}{\mddefault}{\updefault}{\color[rgb]{0,0,0}\normalsize{$w_n=v_{i+1}$}}%
}}}}
\put(8626,-2086){\makebox(0,0)[rb]{\smash{{\SetFigFont{7}{8.4}{\familydefault}{\mddefault}{\updefault}{\color[rgb]{0,0,0}\normalsize{$v_1$}}%
}}}}
\end{picture}%

%% file: Figures/right-insertion-ascent.pstex_t
\begin{picture}(0,0)%
\includegraphics{right-insertion-ascent.pstex}%
\end{picture}%
\setlength{\unitlength}{2368sp}%
\begingroup\makeatletter\ifx\SetFigFont\undefined%
\gdef\SetFigFont#1#2#3#4#5{%
  \reset@font\fontsize{#1}{#2pt}%
  \fontfamily{#3}\fontseries{#4}\fontshape{#5}%
  \selectfont}%
\fi\endgroup%
\begin{picture}(11502,2961)(289,-1939)
\put(3226,-1036){\makebox(0,0)[lb]{\smash{{\SetFigFont{7}{8.4}{\familydefault}{\mddefault}{\updefault}{\color[rgb]{0,0,0}\normalsize{stretch}}%
}}}}
\put(2626,-1186){\makebox(0,0)[rb]{\smash{{\SetFigFont{7}{8.4}{\familydefault}{\mddefault}{\updefault}{\color[rgb]{0,0,0}\normalsize{$m_1$}}%
}}}}
\put(1426,-1861){\makebox(0,0)[rb]{\smash{{\SetFigFont{7}{8.4}{\familydefault}{\mddefault}{\updefault}{\color[rgb]{0,0,0}\normalsize{$a$}}%
}}}}
\put(976,-286){\makebox(0,0)[lb]{\smash{{\SetFigFont{7}{8.4}{\familydefault}{\mddefault}{\updefault}{\color[rgb]{0,0,0}\normalsize{$m_k$}}%
}}}}
\put(7389,-1036){\makebox(0,0)[lb]{\smash{{\SetFigFont{7}{8.4}{\familydefault}{\mddefault}{\updefault}{\color[rgb]{0,0,0}\normalsize{insertion}}%
}}}}
\put(10201,689){\makebox(0,0)[rb]{\smash{{\SetFigFont{7}{8.4}{\familydefault}{\mddefault}{\updefault}{\color[rgb]{0,0,0}\normalsize{$b$}}%
}}}}
\put(9901,-286){\makebox(0,0)[rb]{\smash{{\SetFigFont{7}{8.4}{\familydefault}{\mddefault}{\updefault}{\color[rgb]{0,0,0}\normalsize{$m_k$}}%
}}}}
\put(10074,-1786){\makebox(0,0)[rb]{\smash{{\SetFigFont{7}{8.4}{\familydefault}{\mddefault}{\updefault}{\color[rgb]{0,0,0}\normalsize{$a$}}%
}}}}
\put(901,539){\makebox(0,0)[rb]{\smash{{\SetFigFont{7}{8.4}{\familydefault}{\mddefault}{\updefault}{\color[rgb]{0,0,0}\normalsize{$m_j$}}%
}}}}
\put(10370,262){\makebox(0,0)[rb]{\smash{{\SetFigFont{7}{8.4}{\familydefault}{\mddefault}{\updefault}{\color[rgb]{0,0,0}\normalsize{$w$}}%
}}}}
\put(9451,164){\makebox(0,0)[rb]{\smash{{\SetFigFont{7}{8.4}{\familydefault}{\mddefault}{\updefault}{\color[rgb]{0,0,0}\normalsize{$m_j$}}%
}}}}
\put(3001,689){\makebox(0,0)[rb]{\smash{{\SetFigFont{7}{8.4}{\familydefault}{\mddefault}{\updefault}{\color[rgb]{0,0,0}\normalsize{$w_n$}}%
}}}}
\put(11776,839){\makebox(0,0)[rb]{\smash{{\SetFigFont{7}{8.4}{\familydefault}{\mddefault}{\updefault}{\color[rgb]{0,0,0}\normalsize{$w_n$}}%
}}}}
\put(1876,-61){\makebox(0,0)[lb]{\smash{{\SetFigFont{7}{8.4}{\familydefault}{\mddefault}{\updefault}{\color[rgb]{0,0,0}\normalsize{$e_k$}}%
}}}}
\put(3076,-736){\makebox(0,0)[rb]{\smash{{\SetFigFont{7}{8.4}{\familydefault}{\mddefault}{\updefault}{\color[rgb]{0,0,0}\normalsize{$e_1$}}%
}}}}
\end{picture}%

%% file: Figures/inverse-proof-ascent.pstex_t
\begin{picture}(0,0)%
\includegraphics{inverse-proof-ascent.pstex}%
\end{picture}%
\setlength{\unitlength}{2368sp}%
\begingroup\makeatletter\ifx\SetFigFont\undefined%
\gdef\SetFigFont#1#2#3#4#5{%
  \reset@font\fontsize{#1}{#2pt}%
  \fontfamily{#3}\fontseries{#4}\fontshape{#5}%
  \selectfont}%
\fi\endgroup%
\begin{picture}(7368,3183)(676,-6519)
\put(4951,-6241){\makebox(0,0)[lb]{\smash{{\SetFigFont{7}{8.4}{\familydefault}{\mddefault}{\updefault}{\color[rgb]{0,0,0}\normalsize{$w$}}%
}}}}
\put(8044,-4138){\makebox(0,0)[lb]{\smash{{\SetFigFont{7}{8.4}{\familydefault}{\mddefault}{\updefault}{\color[rgb]{0,0,0}\normalsize{$u$}}%
}}}}
\put(6979,-5788){\makebox(0,0)[lb]{\smash{{\SetFigFont{7}{8.4}{\familydefault}{\mddefault}{\updefault}{\color[rgb]{0,0,0}\normalsize{$v$}}%
}}}}
\put(7351,-5161){\makebox(0,0)[lb]{\smash{{\SetFigFont{7}{8.4}{\familydefault}{\mddefault}{\updefault}\normalsize{Path in $\phi(\pi)$ from $v$ to $u$}}}}}
\end{picture}%

%% file: Figures/ROP-ascent.pstex_t
\begin{picture}(0,0)%
\includegraphics{ROP-ascent.pstex}%
\end{picture}%
\setlength{\unitlength}{1381sp}%
\begingroup\makeatletter\ifx\SetFigFont\undefined%
\gdef\SetFigFont#1#2#3#4#5{%
  \reset@font\fontsize{#1}{#2pt}%
  \fontfamily{#3}\fontseries{#4}\fontshape{#5}%
  \selectfont}%
\fi\endgroup%
\begin{picture}(17034,6437)(1187,-6408)
\put(15001,-5236){\makebox(0,0)[rb]{\smash{{\SetFigFont{5}{6.0}{\rmdefault}{\mddefault}{\updefault}{\color[rgb]{0,0,0}\normalsize{$i_2$}}%
}}}}
\put(13951,-5236){\makebox(0,0)[rb]{\smash{{\SetFigFont{5}{6.0}{\rmdefault}{\mddefault}{\updefault}{\color[rgb]{0,0,0}\normalsize{$i_1$}}%
}}}}
\put(13276,-511){\makebox(0,0)[rb]{\smash{{\SetFigFont{5}{6.0}{\rmdefault}{\mddefault}{\updefault}{\color[rgb]{0,0,0}\normalsize{$\pi(i_2)$}}%
}}}}
\put(13276,-3661){\makebox(0,0)[rb]{\smash{{\SetFigFont{5}{6.0}{\rmdefault}{\mddefault}{\updefault}{\color[rgb]{0,0,0}\normalsize{$\pi(i_1)$}}%
}}}}
\put(18151,-5236){\makebox(0,0)[rb]{\smash{{\SetFigFont{5}{6.0}{\rmdefault}{\mddefault}{\updefault}{\color[rgb]{0,0,0}\normalsize{$i_4$}}%
}}}}
\put(17101,-5236){\makebox(0,0)[rb]{\smash{{\SetFigFont{5}{6.0}{\rmdefault}{\mddefault}{\updefault}{\color[rgb]{0,0,0}\normalsize{$i_3$}}%
}}}}
\put(13276,-1561){\makebox(0,0)[rb]{\smash{{\SetFigFont{5}{6.0}{\rmdefault}{\mddefault}{\updefault}{\color[rgb]{0,0,0}\normalsize{$\pi(i_4)$}}%
}}}}
\put(13276,-4711){\makebox(0,0)[rb]{\smash{{\SetFigFont{5}{6.0}{\rmdefault}{\mddefault}{\updefault}{\color[rgb]{0,0,0}\normalsize{$\pi(i_3)$}}%
}}}}
\put(3184,-1479){\makebox(0,0)[rb]{\smash{{\SetFigFont{5}{6.0}{\rmdefault}{\mddefault}{\updefault}{\color[rgb]{0,0,0}\normalsize{$f_1$}}%
}}}}
\put(1550,-3413){\makebox(0,0)[rb]{\smash{{\SetFigFont{5}{6.0}{\rmdefault}{\mddefault}{\updefault}{\color[rgb]{0,0,0}\normalsize{$v_1$}}%
}}}}
\put(3175,-3948){\makebox(0,0)[rb]{\smash{{\SetFigFont{5}{6.0}{\rmdefault}{\mddefault}{\updefault}{\color[rgb]{0,0,0}\normalsize{$f_2$}}%
}}}}
\put(8834,-1379){\makebox(0,0)[rb]{\smash{{\SetFigFont{5}{6.0}{\rmdefault}{\mddefault}{\updefault}{\color[rgb]{0,0,0}\normalsize{$f_1$}}%
}}}}
\put(7200,-3313){\makebox(0,0)[rb]{\smash{{\SetFigFont{5}{6.0}{\rmdefault}{\mddefault}{\updefault}{\color[rgb]{0,0,0}\normalsize{$v_1$}}%
}}}}
\put(8825,-3848){\makebox(0,0)[rb]{\smash{{\SetFigFont{5}{6.0}{\rmdefault}{\mddefault}{\updefault}{\color[rgb]{0,0,0}\normalsize{$f_2$}}%
}}}}
\put(7277,-2846){\makebox(0,0)[rb]{\smash{{\SetFigFont{5}{6.0}{\rmdefault}{\mddefault}{\updefault}{\color[rgb]{0,0,0}\normalsize{$e_2$}}%
}}}}
\put(4813,-1745){\makebox(0,0)[rb]{\smash{{\SetFigFont{5}{6.0}{\rmdefault}{\mddefault}{\updefault}{\color[rgb]{0,0,0}\normalsize{$v_2$}}%
}}}}
\put(10370,-1709){\makebox(0,0)[rb]{\smash{{\SetFigFont{5}{6.0}{\rmdefault}{\mddefault}{\updefault}{\color[rgb]{0,0,0}\normalsize{$v_2$}}%
}}}}
\put(16126,-4036){\makebox(0,0)[rb]{\smash{{\SetFigFont{5}{6.0}{\rmdefault}{\mddefault}{\updefault}{\color[rgb]{0,0,0}\normalsize{$f_2$}}%
}}}}
\put(14026,-3286){\makebox(0,0)[rb]{\smash{{\SetFigFont{5}{6.0}{\rmdefault}{\mddefault}{\updefault}{\color[rgb]{0,0,0}\normalsize{$v_1$}}%
}}}}
\put(17701,-2011){\makebox(0,0)[rb]{\smash{{\SetFigFont{5}{6.0}{\rmdefault}{\mddefault}{\updefault}{\color[rgb]{0,0,0}\normalsize{$v_2$}}%
}}}}
\put(10726,-1336){\makebox(0,0)[rb]{\smash{{\SetFigFont{5}{6.0}{\rmdefault}{\mddefault}{\updefault}{\color[rgb]{0,0,0}\normalsize{$e_4$}}%
}}}}
\put(10426,-2311){\makebox(0,0)[rb]{\smash{{\SetFigFont{5}{6.0}{\rmdefault}{\mddefault}{\updefault}{\color[rgb]{0,0,0}\normalsize{$e_3$}}%
}}}}
\put(7126,-3736){\makebox(0,0)[rb]{\smash{{\SetFigFont{5}{6.0}{\rmdefault}{\mddefault}{\updefault}{\color[rgb]{0,0,0}\normalsize{$e_1$}}%
}}}}
\put(15451,-1111){\makebox(0,0)[rb]{\smash{{\SetFigFont{5}{6.0}{\rmdefault}{\mddefault}{\updefault}{\color[rgb]{0,0,0}\normalsize{$f_1$}}%
}}}}
\put(8464,-6335){\makebox(0,0)[lb]{\smash{{\SetFigFont{5}{6.0}{\rmdefault}{\mddefault}{\updefault}\normalsize{(b)}}}}}
\put(2871,-6322){\makebox(0,0)[lb]{\smash{{\SetFigFont{5}{6.0}{\rmdefault}{\mddefault}{\updefault}\normalsize{(a)}}}}}
\put(15992,-6310){\makebox(0,0)[rb]{\smash{{\SetFigFont{5}{6.0}{\rmdefault}{\mddefault}{\updefault}\normalsize{(c)}}}}}
\end{picture}%

%% file: Figures/SP.pstex_t
\begin{picture}(0,0)%
\includegraphics{SP.pstex}%
\end{picture}%
\setlength{\unitlength}{1381sp}%
\begingroup\makeatletter\ifx\SetFigFont\undefined%
\gdef\SetFigFont#1#2#3#4#5{%
  \reset@font\fontsize{#1}{#2pt}%
  \fontfamily{#3}\fontseries{#4}\fontshape{#5}%
  \selectfont}%
\fi\endgroup%
\begin{picture}(8062,4681)(1254,-6359)
\put(1576,-6286){\makebox(0,0)[lb]{\smash{{\SetFigFont{5}{6.0}{\rmdefault}{\mddefault}{\updefault}\normalsize{(a)}}}}}
\put(4276,-6286){\makebox(0,0)[lb]{\smash{{\SetFigFont{5}{6.0}{\rmdefault}{\mddefault}{\updefault}\normalsize{(b)}}}}}
\put(7351,-6211){\makebox(0,0)[lb]{\smash{{\SetFigFont{5}{6.0}{\rmdefault}{\mddefault}{\updefault}\normalsize{(c)}}}}}
\put(1951,-3436){\makebox(0,0)[lb]{\smash{{\SetFigFont{5}{6.0}{\rmdefault}{\mddefault}{\updefault}\normalsize{$t$}}}}}
\put(1576,-4486){\makebox(0,0)[lb]{\smash{{\SetFigFont{5}{6.0}{\rmdefault}{\mddefault}{\updefault}\normalsize{$s$}}}}}
\put(3751,-5761){\makebox(0,0)[lb]{\smash{{\SetFigFont{5}{6.0}{\rmdefault}{\mddefault}{\updefault}\normalsize{$s$}}}}}
\put(7051,-5086){\makebox(0,0)[lb]{\smash{{\SetFigFont{5}{6.0}{\rmdefault}{\mddefault}{\updefault}\normalsize{$s$}}}}}
\put(5626,-1861){\makebox(0,0)[lb]{\smash{{\SetFigFont{5}{6.0}{\rmdefault}{\mddefault}{\updefault}\normalsize{$t$}}}}}
\put(8551,-3061){\makebox(0,0)[lb]{\smash{{\SetFigFont{5}{6.0}{\rmdefault}{\mddefault}{\updefault}\normalsize{$t$}}}}}
\put(9301,-4111){\makebox(0,0)[lb]{\smash{{\SetFigFont{5}{6.0}{\rmdefault}{\mddefault}{\updefault}\normalsize{$\check M _2$}}}}}
\put(3526,-4786){\makebox(0,0)[lb]{\smash{{\SetFigFont{5}{6.0}{\rmdefault}{\mddefault}{\updefault}\normalsize{$\check M _1$}}}}}
\put(3976,-2686){\makebox(0,0)[lb]{\smash{{\SetFigFont{5}{6.0}{\rmdefault}{\mddefault}{\updefault}\normalsize{$\check M _2$}}}}}
\put(6226,-4111){\makebox(0,0)[lb]{\smash{{\SetFigFont{5}{6.0}{\rmdefault}{\mddefault}{\updefault}\normalsize{$\check M _1$}}}}}
\end{picture}%

%% file: Figures/K4.pstex_t
\begin{picture}(0,0)%
\includegraphics{K4.pstex}%
\end{picture}%
\setlength{\unitlength}{1973sp}%
\begingroup\makeatletter\ifx\SetFigFont\undefined%
\gdef\SetFigFont#1#2#3#4#5{%
  \reset@font\fontsize{#1}{#2pt}%
  \fontfamily{#3}\fontseries{#4}\fontshape{#5}%
  \selectfont}%
\fi\endgroup%
\begin{picture}(4380,1380)(3511,-1951)
\end{picture}%

%% file: Figures/valid.pstex_t
\begin{picture}(0,0)%
\includegraphics{valid.pstex}%
\end{picture}%
\setlength{\unitlength}{1450sp}%
\begingroup\makeatletter\ifx\SetFigFont\undefined%
\gdef\SetFigFont#1#2#3#4#5{%
  \reset@font\fontsize{#1}{#2pt}%
  \fontfamily{#3}\fontseries{#4}\fontshape{#5}%
  \selectfont}%
\fi\endgroup%
\begin{picture}(3468,4149)(5716,-4574)
\put(6976,-1456){\makebox(0,0)[lb]{\smash{{\SetFigFont{5}{6.0}{\rmdefault}{\mddefault}{\updefault}{\color[rgb]{0,0,0}\normalsize{$v$}}%
}}}}
\put(5716,-4516){\makebox(0,0)[lb]{\smash{{\SetFigFont{5}{6.0}{\rmdefault}{\mddefault}{\updefault}{\color[rgb]{0,0,0}\normalsize{$s$}}%
}}}}
\put(9091,-601){\makebox(0,0)[lb]{\smash{{\SetFigFont{5}{6.0}{\rmdefault}{\mddefault}{\updefault}{\color[rgb]{0,0,0}\normalsize{$t$}}%
}}}}
\put(7561,-916){\makebox(0,0)[lb]{\smash{{\SetFigFont{5}{6.0}{\rmdefault}{\mddefault}{\updefault}{\color[rgb]{0,0,0}\normalsize{$e$}}%
}}}}
\put(6031,-3076){\makebox(0,0)[lb]{\smash{{\SetFigFont{5}{6.0}{\rmdefault}{\mddefault}{\updefault}{\color[rgb]{0,0,0}\normalsize{$s_1$}}%
}}}}
\put(7831,-1726){\makebox(0,0)[lb]{\smash{{\SetFigFont{5}{6.0}{\rmdefault}{\mddefault}{\updefault}{\color[rgb]{0,0,0}\normalsize{$f_1$}}%
}}}}
\end{picture}%

%% file: Figures/equiv_LOP-new.pstex_t
\begin{picture}(0,0)%
\includegraphics{equiv_LOP-new.pstex}%
\end{picture}%
\setlength{\unitlength}{1450sp}%
\begingroup\makeatletter\ifx\SetFigFont\undefined%
\gdef\SetFigFont#1#2#3#4#5{%
  \reset@font\fontsize{#1}{#2pt}%
  \fontfamily{#3}\fontseries{#4}\fontshape{#5}%
  \selectfont}%
\fi\endgroup%
\begin{picture}(14970,4311)(76,-4009)
\put(7336,-2131){\makebox(0,0)[lb]{\smash{{\SetFigFont{5}{6.0}{\rmdefault}{\mddefault}{\updefault}{\color[rgb]{0,0,0}\normalsize{$f_2$}}%
}}}}
\put(10306,-1231){\makebox(0,0)[lb]{\smash{{\SetFigFont{5}{6.0}{\rmdefault}{\mddefault}{\updefault}{\color[rgb]{0,0,0}\normalsize{$f_1$}}%
}}}}
\put(7876,-871){\makebox(0,0)[lb]{\smash{{\SetFigFont{5}{6.0}{\rmdefault}{\mddefault}{\updefault}{\color[rgb]{0,0,0}\normalsize{$t_2$}}%
}}}}
\put(8236,-2851){\makebox(0,0)[lb]{\smash{{\SetFigFont{5}{6.0}{\rmdefault}{\mddefault}{\updefault}{\color[rgb]{0,0,0}\normalsize{$s_1$}}%
}}}}
\put(11251,-16){\makebox(0,0)[lb]{\smash{{\SetFigFont{5}{6.0}{\rmdefault}{\mddefault}{\updefault}{\color[rgb]{0,0,0}\normalsize{$t$}}%
}}}}
\put(9946,-61){\makebox(0,0)[lb]{\smash{{\SetFigFont{5}{6.0}{\rmdefault}{\mddefault}{\updefault}{\color[rgb]{0,0,0}\normalsize{$v$}}%
}}}}
\put(6211,-3166){\makebox(0,0)[lb]{\smash{{\SetFigFont{5}{6.0}{\rmdefault}{\mddefault}{\updefault}{\color[rgb]{0,0,0}\normalsize{$s_2$}}%
}}}}
\put(9631,-736){\makebox(0,0)[lb]{\smash{{\SetFigFont{5}{6.0}{\rmdefault}{\mddefault}{\updefault}{\color[rgb]{0,0,0}\normalsize{$t_3$}}%
}}}}
\put(8686,-1186){\makebox(0,0)[lb]{\smash{{\SetFigFont{5}{6.0}{\rmdefault}{\mddefault}{\updefault}{\color[rgb]{0,0,0}\normalsize{$f_3$}}%
}}}}
\put(8461,-1861){\makebox(0,0)[lb]{\smash{{\SetFigFont{5}{6.0}{\rmdefault}{\mddefault}{\updefault}{\color[rgb]{0,0,0}\normalsize{$s_3$}}%
}}}}
\put(8326,-2266){\makebox(0,0)[lb]{\smash{{\SetFigFont{5}{6.0}{\rmdefault}{\mddefault}{\updefault}{\color[rgb]{0,0,0}\normalsize{$P$}}%
}}}}
\put(1351,-2086){\makebox(0,0)[lb]{\smash{{\SetFigFont{5}{6.0}{\rmdefault}{\mddefault}{\updefault}{\color[rgb]{0,0,0}\normalsize{$f_2$}}%
}}}}
\put(4321,-1186){\makebox(0,0)[lb]{\smash{{\SetFigFont{5}{6.0}{\rmdefault}{\mddefault}{\updefault}{\color[rgb]{0,0,0}\normalsize{$f_1$}}%
}}}}
\put(2251,-2806){\makebox(0,0)[lb]{\smash{{\SetFigFont{5}{6.0}{\rmdefault}{\mddefault}{\updefault}{\color[rgb]{0,0,0}\normalsize{$s_1$}}%
}}}}
\put(5266, 29){\makebox(0,0)[lb]{\smash{{\SetFigFont{5}{6.0}{\rmdefault}{\mddefault}{\updefault}{\color[rgb]{0,0,0}\normalsize{$t$}}%
}}}}
\put(3961,-16){\makebox(0,0)[lb]{\smash{{\SetFigFont{5}{6.0}{\rmdefault}{\mddefault}{\updefault}{\color[rgb]{0,0,0}\normalsize{$v$}}%
}}}}
\put(226,-3121){\makebox(0,0)[lb]{\smash{{\SetFigFont{5}{6.0}{\rmdefault}{\mddefault}{\updefault}{\color[rgb]{0,0,0}\normalsize{$s_2$}}%
}}}}
\put( 91,-3931){\makebox(0,0)[lb]{\smash{{\SetFigFont{5}{6.0}{\rmdefault}{\mddefault}{\updefault}{\color[rgb]{0,0,0}\normalsize{$s$}}%
}}}}
\put(4501,119){\makebox(0,0)[lb]{\smash{{\SetFigFont{5}{6.0}{\rmdefault}{\mddefault}{\updefault}{\color[rgb]{0,0,0}\normalsize{$e$}}%
}}}}
\put(2926,-1321){\makebox(0,0)[lb]{\smash{{\SetFigFont{5}{6.0}{\rmdefault}{\mddefault}{\updefault}{\color[rgb]{0,0,0}\normalsize{$O_1$}}%
}}}}
\put(2251,-736){\makebox(0,0)[lb]{\smash{{\SetFigFont{5}{6.0}{\rmdefault}{\mddefault}{\updefault}{\color[rgb]{0,0,0}\normalsize{$t_2$}}%
}}}}
\put(15031,-736){\makebox(0,0)[lb]{\smash{{\SetFigFont{5}{6.0}{\rmdefault}{\mddefault}{\updefault}{\color[rgb]{0,0,0}\normalsize{$t_3$}}%
}}}}
\put(14086,-1186){\makebox(0,0)[lb]{\smash{{\SetFigFont{5}{6.0}{\rmdefault}{\mddefault}{\updefault}{\color[rgb]{0,0,0}\normalsize{$f_3$}}%
}}}}
\put(13861,-1861){\makebox(0,0)[lb]{\smash{{\SetFigFont{5}{6.0}{\rmdefault}{\mddefault}{\updefault}{\color[rgb]{0,0,0}\normalsize{$s_3$}}%
}}}}
\put(12736,-2131){\makebox(0,0)[lb]{\smash{{\SetFigFont{5}{6.0}{\rmdefault}{\mddefault}{\updefault}{\color[rgb]{0,0,0}\normalsize{$f_2$}}%
}}}}
\put(13276,-871){\makebox(0,0)[lb]{\smash{{\SetFigFont{5}{6.0}{\rmdefault}{\mddefault}{\updefault}{\color[rgb]{0,0,0}\normalsize{$t_2$}}%
}}}}
\put(13636,-2851){\makebox(0,0)[lb]{\smash{{\SetFigFont{5}{6.0}{\rmdefault}{\mddefault}{\updefault}{\color[rgb]{0,0,0}\normalsize{$s_1$}}%
}}}}
\put(11611,-3166){\makebox(0,0)[lb]{\smash{{\SetFigFont{5}{6.0}{\rmdefault}{\mddefault}{\updefault}{\color[rgb]{0,0,0}\normalsize{$s_2$}}%
}}}}
\put(5131,-3886){\makebox(0,0)[rb]{\smash{{\SetFigFont{5}{6.0}{\rmdefault}{\mddefault}{\updefault}\normalsize{(a)}}}}}
\put(9541,-3886){\makebox(0,0)[rb]{\smash{{\SetFigFont{5}{6.0}{\rmdefault}{\mddefault}{\updefault}\normalsize{(b)}}}}}
\put(14671,-3886){\makebox(0,0)[rb]{\smash{{\SetFigFont{5}{6.0}{\rmdefault}{\mddefault}{\updefault}\normalsize{(c)}}}}}
\end{picture}%

%% file: Figures/involution-ascent.pstex_t
\begin{picture}(0,0)%
\includegraphics{involution-ascent.pstex}%
\end{picture}%
\setlength{\unitlength}{1973sp}%
\begingroup\makeatletter\ifx\SetFigFont\undefined%
\gdef\SetFigFont#1#2#3#4#5{%
  \reset@font\fontsize{#1}{#2pt}%
  \fontfamily{#3}\fontseries{#4}\fontshape{#5}%
  \selectfont}%
\fi\endgroup%
\begin{picture}(9443,4374)(3139,-3973)
\end{picture}%